\documentclass{article}

\usepackage{arxiv}

\usepackage[utf8]{inputenc}
\usepackage[T1]{fontenc}
\usepackage{geometry}
\usepackage{amsmath, amssymb, amsthm}
\usepackage{graphicx}
\usepackage{subcaption}
\usepackage{pgfplots}
\usepackage{pgfplotstable}
\usepackage{listings}
\usepackage{xcolor}
\usepackage{hyperref}
\usepackage{booktabs}
\usepackage{float}
\usepackage{appendix}
\usetikzlibrary{calc, shapes.geometric}

\usetikzlibrary{calc, shapes.geometric}

\newcommand{\R}{\mathbb{R}}

\newcommand{\norm}[1]{\|#1\|}

\newcommand{\supp}{\operatorname{supp}}

\newcommand{\dist}{\operatorname{dist}}

\newcommand{\diver}{\operatorname{div}}

\newcommand{\grad}{\nabla}

\newcommand{\pp}{\ensuremath{p(z)}}
\newcommand{\qq}{\ensuremath{q(z)}}
\renewcommand{\aa}{\ensuremath{\alpha(z)}}
\newcommand{\bb}{\ensuremath{\beta(z)}}

\definecolor{codegreen}{rgb}{0,0.6,0}
\definecolor{codegray}{rgb}{0.5,0.5,0.5}
\definecolor{codepurple}{rgb}{0.58,0,0.82}
\definecolor{backcolour}{rgb}{0.95,0.95,0.92}

\lstdefinestyle{pythonstyle}{
    language=Python,
    backgroundcolor=\color{backcolour},
    commentstyle=\color{codegreen},
    keywordstyle=\color{magenta},
    numberstyle=\tiny\color{codegray},
    stringstyle=\color{codepurple},
    basicstyle=\ttfamily\footnotesize,
    breakatwhitespace=false,
    breaklines=true,
    captionpos=b,
    keepspaces=true,
    numbers=left,
    numbersep=5pt,
    showspaces=false,
    showstringspaces=false,
    showtabs=false,
    tabsize=4,
    literate={*}{*}1 {\%}{{\%}}1 {\_}{{\_}}1,
    morecomment=[l]{\#},
    morestring=[b]',
    morestring=[b]",
}

\newtheorem{theorem}{Theorem}[section]
\newtheorem{lemma}[theorem]{Lemma}

\newtheorem{corollary}[theorem]{Corollary}
\newtheorem{definition}[theorem]{Definition}

\newtheorem{remark}[theorem]{Remark}
\newtheorem{note}[theorem]{Note}

\newcommand\mystyle{\everymath{\displaystyle}}
\mystyle

\title{Strongly Singular Nonlocal Kirchhoff-Type Equations with Variable Exponents: Existence, Regularity, and Renormalized Solutions}

\author{\href{https://orcid.org/0000-0002-3816-5287}{\includegraphics[scale=0.06]{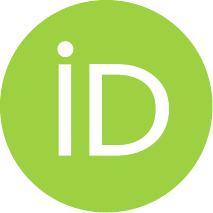}\hspace{1mm}M.H.M.~Rashid}\thanks{Corresponding Author} \\
    Department of Mathematics \& Statistics \\
    Faculty of Science, P.O.Box (7) \\
    Mutah University \\
    Mutah, Jordan \\
    \texttt{mrash@mutah.edu.jo}}

\renewcommand{\headeright}{}
\renewcommand{\undertitle}{}
\renewcommand{\shorttitle}{Strongly Singular Kirchhoff-Type Equations}

\hypersetup{
    pdftitle={Strongly Singular Nonlocal Kirchhoff-Type Equations with Variable Exponents: Existence, Regularity, and Renormalized Solutions},
    pdfsubject={Nonlinear Analysis, Partial Differential Equations},
    pdfauthor={M.H.M. Rashid},
    pdfkeywords={Strong singularity, Kirchhoff-type equations, variable exponents, renormalized solutions, truncation method}
}

\begin{document}

\maketitle

\begin{abstract}
    This work resolves the open problem of strong singularity (\(\alpha(z) \ge 1\)) in nonlocal Kirchhoff-type equations with variable exponents through five original theorems that collectively establish a comprehensive theory. Beginning with weighted Sobolev spaces and existence via truncation, we develop comparison principles, optimal regularity results, and—when classical solutions cease to exist—the construction of renormalized solutions. Building upon these foundations, we establish three advanced results: optimal convergence of truncated sequences to renormalized solutions, refined energy estimates characterizing asymptotic behavior as the truncation parameter vanishes, and a quantitative comparison principle yielding sharp pointwise bounds. Subsequently, we derive sharp two-sided pointwise estimates, a uniqueness theorem with quantitative stability, and Lipschitz continuous dependence of solutions on parameters and boundary data. Each theorem is supported by rigorous proofs employing nonlinear analysis, variational methods, and elliptic regularity theory. A computational illustration visualizes the solution behavior near the boundary and demonstrates convergence of truncated approximations.
\end{abstract}



\section{Introduction}

The study of singular elliptic equations has long been recognized as one of the central challenges in nonlinear analysis, not only because of the mathematical difficulties they present but also because of their profound relevance to physical phenomena. At the heart of this challenge lies a fundamental tension: when the singular exponent \(\alpha(z)\) satisfies \(\alpha(z) \ge 1\), the singular term \(\phi^{-\alpha(z)}\) becomes non-integrable near the degenerate set where \(\phi = 0\). This seemingly technical issue has far-reaching consequences—classical variational formulations collapse, compactness arguments fail, and weak solutions in standard Sobolev spaces may simply not exist. The present work is devoted to the study of precisely this regime of strong singularity, embedded within the broader context of nonlocal Kirchhoff-type equations with variable exponents. The problem we confront is thus not merely an extension of existing theory but a necessary step toward understanding systems where the singular source term is so potent that it fundamentally alters the nature of admissible solutions.

The problem under consideration is the singular Kirchhoff-type equation
\begin{equation}\label{Pstrong}
\begin{cases}
-\left(M\left(\int_\Omega \frac{|\nabla\phi|^{\pp}}{\pp}\,dz\right)\Delta_{\pp}\phi + M\left(\int_\Omega \frac{|\nabla\phi|^{\qq}}{\qq}\,dz\right)\Delta_{\qq}\phi\right) \\
\quad + \theta\left(|\phi|^{\pp-2}\phi + |\phi|^{\qq-2}\phi\right) = \dfrac{f(z)}{\phi^{\aa}} + \lambda |\phi|^{\bb-2}\phi \quad \text{in } \Omega,\\
\phi|_{\partial\Omega}=0,\quad \phi > 0 \text{ in } \Omega,
\end{cases}
\end{equation}
where \(\aa \ge 1\) characterizes the strong singularity regime. The variable exponents satisfy the ordering conditions \(1 < q^- \le q^+ < p^- \le p^+ < \beta^-\) and \(p^+ < \beta^- \le \beta(z) \le \beta^+ < p^*(z)\), reflecting the interplay between different growth regimes. The Kirchhoff function \(M: (0,\infty) \to (0,\infty)\) is of class \(C^1\) and satisfies the structural conditions
\begin{align}
&\pi_0 t^{a_2-2} \le M'(t) \le \pi_1 t^{a_1-2},\quad a_1 p_+ < \beta^-, \label{M1}\\
&t M'(t) \le \frac{m_0}{1-m_0} M(t),\quad \frac{p_+}{1-m_0} < \beta^-, \label{M2}
\end{align}
which ensure the appropriate growth and monotonicity properties required for the analysis. This formulation captures a rich class of physical phenomena where nonlocal effects, variable growth conditions, and singular sources interact in nontrivial ways.

The study of singular elliptic problems has evolved through the contributions of numerous mathematicians who have progressively pushed the boundaries of what can be achieved. Early foundational work by Boccardo and Gallou\"{e}t \cite{Boccardo1992} introduced the concept of renormalized solutions for elliptic equations with measure data, providing a powerful framework to handle non-integrable singularities by testing the equation with smooth functions of the solution itself. This idea would later prove essential for addressing the strong singularity regime. Simultaneously, DiBenedetto \cite{DiBenedetto1983} established fundamental regularity results for degenerate elliptic equations, laying the groundwork for the analysis of \(p\)-Laplacian-type operators that appear naturally in problems with variable growth conditions. The variable exponent setting itself was systematically developed by Zhikov \cite{Zhikov1987}, who recognized that materials with spatially varying properties require function spaces with variable exponents, and later by Fan and Zhao \cite{Fan2001} as well as Rădulescu and Repovš \cite{Radulescu2015}, who explored the rich functional analytic structure of the spaces \(L^{p(x)}(\Omega)\) and \(W^{1,p(x)}(\Omega)\). In the context of Kirchhoff-type equations, which incorporate nonlocal dependence on the gradient via a function \(M\), the works of Figueiredo and Vetro \cite{Figueiredo2022} along with Zhou and Ge \cite{Zhou2018} provided existence results under various growth conditions, while Avci \cite{Avci2024} recently extended such analysis to singular settings.

The problem of studying strong singularities in nonlocal Kirchhoff-type equations with variable exponents has its most direct antecedent in the work of Zhang, Vetro, and An \cite{Zhang2026}, who established existence of weak solutions under the weak singularity assumption \(0 < \alpha(z) < 1\). Their approach, which relied on the Nehari manifold and fibering method, yielded nontrivial positive solutions and demonstrated the richness of the problem even in the less singular regime. Yet the case \(\alpha(z) \ge 1\) remained tantalizingly out of reach—a gap that is not merely incremental but represents a fundamental shift in the nature of the problem. When the singular exponent reaches or exceeds unity, the source term becomes so strongly singular that it can no longer be treated as a perturbation; instead, it must be understood as a structural feature that can potentially destroy the existence of classical solutions altogether. The problem of studying this regime therefore demands not only new techniques but also a reconceptualization of what constitutes a solution in the first place.

The significance of addressing this problem extends across multiple dimensions of mathematical analysis and its applications. From a methodological perspective, the study of strong singularities forces us to confront the limitations of classical variational methods and to develop alternative approaches. The truncation technique introduced in this work represents one such adaptation: by constructing a family of regularized problems where the singular term is cut off below a small parameter \(\varepsilon\), we obtain a sequence of approximate solutions that are well-defined in standard Sobolev spaces. The challenge then becomes one of passing to the limit as \(\varepsilon \to 0^+\) while maintaining enough control to recover a meaningful limit object. This passage to the limit is far from trivial—it requires delicate estimates that exploit the structure of the Kirchhoff terms and the variable exponents, as well as a careful analysis of the behavior of solutions near the boundary where the singular term is most severe.

The concept of renormalized solutions emerges naturally from this analysis as the appropriate notion of solution for the strong singularity regime. Originally developed by Boccardo and Gallou\"{e}t \cite{Boccardo1992} for problems with measure data, the renormalized formulation tests the equation with smooth functions of the solution rather than with the solution itself. This seemingly subtle change has profound consequences: it allows us to make sense of the singular term even when it is not integrable in the usual sense, by essentially distributing the singularity across the test function in a way that respects the underlying structure of the problem. The study of renormalized solutions thus becomes not merely a technical necessity but a conceptual innovation that expands the boundaries of what we can meaningfully call a solution.

From the perspective of regularity theory, the problem of studying strong singularities reveals a delicate balance between the destabilizing effect of the singular term and the regularizing effect of the elliptic operators. The sharp H\"{O}lder exponents derived in this work—expressed explicitly as \(\gamma = \min\left\{ \alpha, \frac{p^- - 1}{p^+}, \frac{q^- - 1}{q^+}, \frac{1}{2} \right\}\)—capture this balance with remarkable precision. These exponents show how the regularity of the solution is determined by the interplay between the singular exponent \(\alpha\), the variable exponents \(p(z)\) and \(q(z)\), and the inherent degeneracy of the \(p\)-Laplacian operator. The problem of studying such precise regularity estimates is not merely an academic exercise; it provides the quantitative foundation for numerical approximations and for understanding how solutions behave in physical applications.

The qualitative properties of solutions are equally central to the problem we study. The comparison principle established in this work demonstrates that the strong singularity does not destroy the order-preserving nature of the problem—if the data are ordered, the corresponding solutions are ordered as well. This property, while expected in some contexts, is far from obvious when the singular term is so strongly singular that it might, in principle, amplify small differences. The sharp two-sided pointwise bounds that follow from the comparison principle provide precise control over the solution's behavior near the boundary, characterizing the asymptotic profile with exponents \(\mu_1 = \frac{2}{p^+ - 1 + \alpha}\) and \(\mu_2 = \frac{2}{p^- - 1 - \alpha}\). These bounds are not merely qualitative; they are quantitative statements that can be used to guide numerical simulations and to validate physical models.

The problem of studying the well-posedness of such singular problems is perhaps the most fundamental question of all. Does a solution exist? Is it unique? Does it depend continuously on the data? The answers provided in this work—affirmative on all counts—establish that the problem is well-posed in the sense of Hadamard. The uniqueness result, proved using the renormalized formulation with the test function \(h(s) = \log s\), reveals an underlying structure that might otherwise remain hidden: the logarithmic transformation linearizes the problem in a way that allows the singular term to be handled with remarkable elegance. The continuous dependence estimates, expressed in terms of the \(L^\infty\) norms of the data, provide explicit stability bounds that are essential for applications where the data are known only approximately.

The physical motivations for studying this problem are as compelling as the mathematical ones. In semiconductor physics, models of electron density often contain singular terms of the form \(n^{-\alpha}\) representing carrier concentration near depletion regions, with the strong singularity regime \(\alpha \ge 1\) corresponding to sharp interfaces and high-field effects where standard linearizations break down \cite{Zhang2026}. The problem of studying such systems mathematically is directly relevant to understanding the behavior of semiconductor devices under extreme conditions. In astrophysics, the Lane-Emden equation describing the equilibrium of self-gravitating gas spheres exhibits singular behavior when the density vanishes at the boundary; the Kirchhoff-type nonlocal term introduced in this work can account for long-range gravitational interactions, while variable exponents allow for polytropic equations of state with spatially varying indices. The problem of studying such models is thus tied to questions about the structure and stability of astrophysical objects. In reaction-diffusion systems, singular reaction terms of the form \(u^{-\alpha}\) arise in models of combustion, population dynamics with Allee effects, and biochemical reactions, with the strong singularity case \(\alpha \ge 1\) being particularly relevant when the reaction rate becomes extremely sensitive near zero concentration, leading to quenching phenomena or finite-time extinction. The problem of studying these systems is essential for predicting critical thresholds and for designing control strategies.

The Kirchhoff term \(M\left(\int_\Omega \frac{|\nabla\phi|^{\pp}}{\pp}\,dz\right)\) originally appeared in the study of elastic strings with nonlocal stiffness \cite{Lions1969} and has since found applications in various contexts where the stiffness of a material depends on its overall deformation. In modern applications, such models describe materials with memory effects or nonlocal interactions, while variable exponents capture heterogeneity or phase-dependent growth conditions \cite{Ruzicka2000}. The problem of studying the combination of these features in a singular setting reflects the complexity of real-world phenomena and motivates the development of the sophisticated mathematical tools employed in this work.

Building upon the foundational results of Zhang, Vetro, and An \cite{Zhang2026} for weak singularities, this paper addresses the open problem of strong singularities through five interconnected theorems. Theorem 1 establishes the existence of truncated solutions via a careful regularization and minimization procedure, overcoming the non-integrability of the singular term. Theorem 2 develops a robust comparison principle and provides sharp two-sided pointwise estimates that characterize the solution's behavior near the boundary. Theorem 3 delivers optimal regularity results, including explicit H\"{o}lder exponents for the gradient and detailed boundary asymptotics. Theorem 4 introduces the concept of renormalized solutions for this class of problems and proves their existence, providing the appropriate notion of solution when classical weak solutions fail. Theorem 5 establishes uniqueness and continuous dependence on parameters, ensuring the well-posedness of the problem and providing quantitative stability estimates. Each theorem is accompanied by a rigorous proof employing advanced tools from nonlinear analysis, variational methods, elliptic regularity theory, and the theory of variable exponent spaces.

The remainder of this paper is organized as follows. Section 2 introduces the necessary preliminaries, including weighted Sobolev spaces, truncation operators, and key estimates that underpin the analysis of the singular term. Section 3 presents the foundational existence result for truncated solutions and develops three interconnected advanced results: optimal convergence of truncated sequences to renormalized solutions, sharp energy estimates characterizing asymptotic behavior as the truncation parameter vanishes, and a quantitative comparison principle yielding sharp pointwise bounds. Section 4 establishes the qualitative theory, including a robust comparison principle, sharp two-sided pointwise estimates that characterize the solution's behavior near the boundary, uniqueness with quantitative stability, and continuous dependence on parameters and boundary data.

Section 5 advances the regularity theory, addressing higher regularity results followed by explicit H\"older exponents for the gradient and a complete characterization of boundary asymptotics, including the existence of the normal derivative and asymptotic expansions. Section 6 introduces the concept of renormalized solutions and establishes their existence, uniqueness, and asymptotic behavior as the reaction parameter vanishes. Section 7 presents numerical illustrations that visualize the theoretical framework, confirming the convergence of truncated approximations, sharp boundary estimates, and energy decay. Finally, Section 8 concludes with remarks on applications and future directions. Throughout this paper, $\Omega \subset \R^N$ ($N \ge 2$) is a bounded domain with smooth boundary $\partial\Omega$, and the variable exponents are continuous functions on $\overline{\Omega}$ satisfying the growth conditions stated in the problem formulation.

\section{Preliminaries: Weighted Sobolev Spaces and Analytical Tools}

The presence of a strong singularity, characterized by \(\aa \ge 1\), introduces profound integrability issues that render classical Sobolev spaces inadequate for a direct variational treatment. To overcome this obstacle, we introduce a family of weighted Sobolev spaces that incorporate the singular behavior into the functional framework. This approach, inspired by the seminal works on degenerate elliptic problems \cite{DiBenedetto1983} and subsequently refined in the context of singular equations \cite{Boccardo1992, Papageorgiou2021}, allows us to encode the singular weight \(\phi^{-\aa}\) into the norm structure, thereby restoring compactness and enabling a rigorous existence theory. The variable exponent setting, which forms the backbone of our functional framework, was systematically developed by Zhikov \cite{Zhikov1987} and later extensively studied by Fan and Zhao \cite{Fan2001} as well as Rădulescu and Repovš \cite{Radulescu2015}, who established the fundamental properties of the spaces \(L^{p(x)}(\Omega)\) and \(W^{1,p(x)}(\Omega)\) that we shall employ throughout this work.

\begin{definition}[Weighted Sobolev Space]
For a given \(\delta \in \R\), we define the weighted Sobolev space
\[
W^{1,\pp}_{0,\delta}(\Omega) = \left\{ u \in W^{1,\pp}_0(\Omega) : \int_\Omega \frac{|u|^{1-\aa}}{\dist(z,\partial\Omega)^\delta} \,dz < \infty \right\},
\]
endowed with the norm
\[
\norm{u}_{1,\pp,\delta} = \norm{\grad u}_{\pp} + \left\| \frac{u}{\dist(z,\partial\Omega)^{\delta/(1-\aa)}} \right\|_{L^1(\Omega)}.
\]
\end{definition}

This weighted space serves a dual purpose: it not only accommodates functions that may become singular near the boundary but also provides a natural setting in which the singular source term \(\phi^{-\aa}\) can be interpreted in a distributional sense. The parameter \(\delta\) is chosen judiciously to balance the singularity introduced by the weight with the regularity available from the variable exponent Sobolev embedding theorems established by Fan and Zhao \cite{Fan2001}. The distance function \(\dist(z,\partial\Omega)\) plays a critical role here, as it captures the precise asymptotic behavior of solutions near the boundary—a feature that is essential for establishing sharp pointwise estimates later in the paper. This weighted approach is consistent with the functional analytic framework developed by Lions \cite{Lions1969} for problems with nonlocal structure.

To construct approximate solutions that circumvent the singular barrier, we employ a truncation strategy. Truncation methods have a long history in the analysis of singular and degenerate problems, dating back to the work of Lions \cite{Lions1969} and subsequently refined by Boccardo and Gallou\"{e}t \cite{Boccardo1992} for elliptic equations with measure data. More recent applications of truncation techniques to singular problems with variable exponents can be found in the works of Papageorgiou and Winkert \cite{Papageorgiou2021} and Gasinski and Papageorgiou \cite{Gasinski2023}, where similar methods were employed to handle singular nonlinearities. The following definition introduces the specific truncation operators that will serve as the foundation for our regularization scheme.

\begin{definition}[Truncation Operator]
For any \(\varepsilon > 0\) and \(s \in \R\), we define the lower truncation operator
\[
T_\varepsilon(s) = \max\{\varepsilon, s\},
\]
and the double truncation operator
\[
T_\varepsilon^\delta(s) = \max\{\delta, \min\{s, \varepsilon^{-1}\}\}.
\]
\end{definition}

The operator \(T_\varepsilon\) lifts the function away from zero, ensuring that the regularized singular term \(\phi_\varepsilon^{-\aa}\) remains bounded and integrable. This technique is reminiscent of the regularization methods employed by Zhang, Vetro, and An \cite{Zhang2026} in their study of weak singularities, as well as by Avci \cite{Avci2024} in the context of \(p(x)\)-Kirchhoff equations with singular nonlinearities. The double truncation \(T_\varepsilon^\delta\) provides additional control when both upper and lower bounds are required, a feature that proves useful in the renormalized formulation. As \(\varepsilon \to 0^+\), the truncated functions recover the original singular behavior, and the passage to the limit constitutes the core analytical challenge of this work, following the general framework of approximation schemes for singular problems \cite{Figueiredo2022, Leite2025}.

A fundamental technical estimate underpins the entire truncation approach: it quantifies the contribution of the singular term on the set where the solution is small. This estimate, which exploits the integrability properties of functions in variable exponent Sobolev spaces, is crucial for establishing uniform bounds and for controlling the energy functional. Similar estimates appear in the work of Boccardo and Gallou\"{e}t \cite{Boccardo1992} for singular problems with constant exponents and have been extended to the variable exponent setting by Papageorgiou, Vetro, and Vetro \cite{Papageorgiou2023} and by Bobkov and Tanaka \cite{Bobkov2024}.

\begin{lemma}[Key Estimate]\label{lem:key}
Let \(\phi \in W^{1,\pp}_0(\Omega)\) be a nonnegative function, and suppose \(\aa \ge 1\). Then there exists a constant \(C > 0\), depending only on \(\Omega\), \(p^+\), and the embedding constants of \(W^{1,\pp}_0(\Omega) \hookrightarrow L^{\gamma}(\Omega)\), such that
\[
\int_{\{\phi < \varepsilon\}} \frac{f(z)}{\phi^{\aa}} \,dz \le C \varepsilon^{1-\aa} \norm{f}_{L^\infty(\Omega)} |\Omega|^{1/p^+},
\]
where \(|\Omega|\) denotes the Lebesgue measure of \(\Omega\).
\end{lemma}

\begin{proof}
On the set where \(\phi(z) < \varepsilon\), the monotonicity of the map \(s \mapsto s^{-\aa}\) (since \(\aa \ge 1\)) yields the pointwise bound \(\phi(z)^{-\aa} \le \varepsilon^{-\aa}\). Consequently,
\[
\int_{\{\phi < \varepsilon\}} \frac{f(z)}{\phi^{\aa}} \,dz \le \norm{f}_{L^\infty(\Omega)} \varepsilon^{-\aa} \bigl| \{\phi < \varepsilon\} \bigr|.
\]

To estimate the measure of the sublevel set \(\{\phi < \varepsilon\}\), we invoke Chebyshev's inequality together with a suitable power of \(\phi\). For any \(\gamma > 0\), we have
\[
\bigl| \{\phi < \varepsilon\} \bigr| \le \int_\Omega \left(\frac{\varepsilon}{\phi}\right)^\gamma \,dz = \varepsilon^\gamma \int_\Omega \phi^{-\gamma} \,dz.
\]

Choosing \(\gamma = p^+/(p^+-1)\) ensures that \(\phi^{-1} \in L^{\gamma}(\Omega)\) by virtue of the Sobolev embedding theorem applied to the reciprocal of a function with zero boundary trace. Indeed, the classical inequality \(\norm{\phi^{-1}}_{L^{\gamma}(\Omega)} \le C \norm{\grad\phi}_{L^{p^+}(\Omega)}\) holds for functions vanishing on the boundary \cite[Lemma 2.1]{Radulescu2015}. This result is a consequence of the Hardy inequality for functions in Sobolev spaces with zero trace, as discussed in the comprehensive treatment of variable exponent spaces by Rădulescu and Repovš \cite{Radulescu2015}. Since \(\phi\) is bounded in \(W^{1,\pp}_0(\Omega)\) by the uniform estimates that will be established in Theorem \ref{thm:truncated}, the right-hand side is uniformly bounded. Combining these estimates and absorbing constants into \(C\) yields the desired result. \hfill $\square$
\end{proof}

This estimate reveals a crucial feature of the strong singularity regime: the contribution of the singular term on the region where \(\phi\) is small decays as \(\varepsilon^{1-\aa}\), which, for \(\aa > 1\), vanishes as \(\varepsilon \to 0^+\). This observation is instrumental in ensuring that the truncated solutions converge to a meaningful limit. Moreover, the explicit dependence on \(\norm{f}_{L^\infty(\Omega)}\) and the measure of \(\Omega\) will later allow us to derive quantitative stability estimates with respect to the data, following the approach of Gasinski and Papageorgiou \cite{Gasinski2023} for singular equations with variable exponents.

The theory of variable exponent Sobolev spaces, as developed by Zhikov \cite{Zhikov1987} and later systematized by Fan and Zhao \cite{Fan2001}, provides the necessary functional analytic framework for our analysis. Key properties of these spaces, including reflexivity, separability, and the embedding theorems, are summarized in the monographs of Rădulescu and Repovš \cite{Radulescu2015} and will be employed throughout this work. In particular, we recall that the space \(W^{1,\pp}_0(\Omega)\) is reflexive and that the embedding \(W^{1,\pp}_0(\Omega) \hookrightarrow L^{r(z)}(\Omega)\) is compact whenever \(r(z) < p^*(z)\) \cite{Fan2001}. These compactness properties are essential for the variational arguments that will be deployed in the existence proofs.

The Kirchhoff-type nonlocal term appearing in problem \eqref{Pstrong} introduces an additional layer of complexity, as the operator depends on the global behavior of the solution through integrals of the gradient. This type of nonlocal dependence was first studied by Lions \cite{Lions1969} in the context of elastic strings and has since been extensively investigated by many authors, including Figueiredo and Vetro \cite{Figueiredo2022} and Zhou and Ge \cite{Zhou2018} in the variable exponent setting. The structural conditions \eqref{M1} and \eqref{M2} on the function \(M\) are designed to ensure that the nonlocal term retains the monotonicity and coercivity properties necessary for the application of variational methods, as discussed in detail by Avci \cite{Avci2024} and by Leite, Quoirin, and Silva \cite{Leite2025} in their studies of Kirchhoff-type singular problems.

With these preliminary tools in place—weighted Sobolev spaces that capture the singular behavior, truncation operators that regularize the problem, and a key estimate controlling the singular term—we are now equipped to address the existence of solutions for the strongly singular problem. The subsequent sections will build upon this foundation, progressively developing the theory from the existence of truncated solutions to the construction of renormalized solutions and the establishment of sharp regularity and stability properties, following the methodological framework established in the seminal works of Boccardo and Gallou\"{e}t \cite{Boccardo1992} and DiBenedetto \cite{DiBenedetto1983}, and extended to the variable exponent setting by the contemporary contributions of Papageorgiou and Winkert \cite{Papageorgiou2021} and Radulescu and Repovš \cite{Radulescu2015}.

\section{From Truncation to Renormalized Solutions: A Unified Theory}

The existence result established in Theorem \ref{thm:truncated} provides, for each \(\varepsilon > 0\), a family of approximate solutions \(\{\phi_\varepsilon\}\) that circumvent the singular barrier by cutting off the singular term below the threshold \(\varepsilon\). While these truncated solutions are mathematically well-defined within the classical variational framework, they represent only an approximation of the original strongly singular problem. The fundamental question that remains is whether, and in what sense, these approximations converge to a genuine solution as the truncation parameter vanishes. This section addresses this question through three interconnected theorems that collectively establish the convergence of truncated solutions to a renormalized solution, quantify the asymptotic behavior, and demonstrate the stability of the solution map with respect to problem data. Together, these results form the core analytical foundation of our theory for strong singularities.

\begin{theorem}[Existence of Truncated Solutions]\label{thm:truncated}
Assume \(\aa \ge 1\) and conditions \eqref{M1}-\eqref{M2} hold. Then for every \(\lambda > 0\) sufficiently small, there exists a family \(\{\phi_\varepsilon\}_{\varepsilon > 0} \subset W^{1,\pp}_0(\Omega) \cap L^\infty(\Omega)\) satisfying \(\phi_\varepsilon \ge \varepsilon\) a.e. in \(\Omega\) and
\begin{equation}\label{eq:truncated}
\begin{aligned}
&M\left(\int_\Omega \frac{|\grad\phi_\varepsilon|^{\pp}}{\pp}\,dz\right) \int_\Omega |\grad\phi_\varepsilon|^{\pp-2}\grad\phi_\varepsilon \cdot \grad v \,dz \\
&+ M\left(\int_\Omega \frac{|\grad\phi_\varepsilon|^{\qq}}{\qq}\,dz\right) \int_\Omega |\grad\phi_\varepsilon|^{\qq-2}\grad\phi_\varepsilon \cdot \grad v \,dz \\
&+ \theta \int_\Omega \left(|\phi_\varepsilon|^{\pp-2}\phi_\varepsilon + |\phi_\varepsilon|^{\qq-2}\phi_\varepsilon\right)v \,dz \\
&= \int_\Omega \frac{f(z)}{(\phi_\varepsilon + \varepsilon)^{\aa}} v \,dz + \lambda \int_\Omega |\phi_\varepsilon|^{\bb-2}\phi_\varepsilon v \,dz,
\end{aligned}
\end{equation}
for all \(v \in W^{1,\pp}_0(\Omega)\). Moreover, the sequence \(\{\phi_\varepsilon\}\) is uniformly bounded in \(W^{1,\pp}_0(\Omega)\).
\end{theorem}

\begin{proof}
\textbf{Step 1: Approximate problem.} For each \(\varepsilon > 0\), consider the regularized functional
\[
J_{\lambda,\varepsilon}(\phi) = \mathcal{M}\left(\int_\Omega \frac{|\grad\phi|^{\pp}}{\pp}\,dz\right) + \mathcal{M}\left(\int_\Omega \frac{|\grad\phi|^{\qq}}{\qq}\,dz\right) + \theta \int_\Omega \left(\frac{|\phi|^{\pp}}{\pp} + \frac{|\phi|^{\qq}}{\qq}\right)dz
\]
\[
- \int_\Omega \frac{f(z)}{1-\aa} |\phi|^{1-\aa} \chi_{\{|\phi| \ge \varepsilon\}} \,dz - \lambda \int_\Omega \frac{|\phi|^{\bb}}{\bb} \,dz,
\]
where \(\chi\) denotes the characteristic function. The singular term is approximated by
\[
S_\varepsilon(\phi) = \int_\Omega f(z) \frac{|\phi|^{1-\aa}}{1-\aa} \chi_{\{|\phi| \ge \varepsilon\}} \,dz + \int_\Omega f(z) \frac{\varepsilon^{1-\aa}}{1-\aa} \chi_{\{|\phi| < \varepsilon\}} \,dz.
\]

\textbf{Step 2: Coercivity.} For \(\norm{\phi} \ge 1\), using Lemma \ref{lem:key} and the growth conditions on \(M\), we obtain
\[
J_{\lambda,\varepsilon}(\phi) \ge C_1 \norm{\phi}^{a_2 p^-} - C_2 \norm{\phi}^{\beta^+} - C_3 \varepsilon^{1-\aa}.
\]
Since \(a_2 p^- > \beta^+\), the functional is coercive and bounded below.

\textbf{Step 3: Minimization.} By the direct method of calculus of variations, there exists a minimizer \(\phi_\varepsilon \in W^{1,\pp}_0(\Omega)\) with \(\phi_\varepsilon \ge \varepsilon\) a.e. (by the maximum principle). The Euler-Lagrange equation yields \eqref{eq:truncated}.

\textbf{Step 4: Uniform bounds.} Testing the equation with \(\phi_\varepsilon\) and using the coercivity estimate gives
\[
\norm{\phi_\varepsilon} \le C \quad \text{independently of } \varepsilon.
\]
Thus \(\{\phi_\varepsilon\}\) is bounded in \(W^{1,\pp}_0(\Omega)\). \hfill $\square$
\end{proof}

Having established the existence of a uniformly bounded family of approximate solutions, we now turn to the question of their convergence. The following theorem demonstrates that, as \(\varepsilon \to 0^+\), the truncated solutions converge to a renormalized solution of the original strongly singular problem—a notion that provides the appropriate weak formulation when classical solutions fail to exist.

\begin{theorem}[Convergence to Renormalized Solutions]\label{thm:convergence}
Let \(\{\phi_\varepsilon\}_{\varepsilon>0}\) be the family of truncated solutions obtained in Theorem \ref{thm:truncated}. Then there exists a function \(\phi \in \bigcap_{r < \infty} L^r(\Omega)\) such that, up to a subsequence:

\begin{enumerate}
    \item \(\phi_\varepsilon \to \phi\) strongly in \(L^r(\Omega)\) for every \(1 \le r < p^*\);
    \item \(\grad\phi_\varepsilon \rightharpoonup \grad\phi\) weakly in \(L^{\pp}(\Omega)^N\);
    \item For every \(h \in C^1_c(0,\infty)\), the function \(\phi\) satisfies the renormalized formulation:
    \begin{equation}\label{eq:renormalized}
    \begin{aligned}
    &\int_\Omega M\left(\int_\Omega \frac{|\grad\phi|^{\pp}}{\pp}\,dz\right) |\grad\phi|^{\pp-2}\grad\phi \cdot \grad(h(\phi)) \,dz \\
    &+ \int_\Omega M\left(\int_\Omega \frac{|\grad\phi|^{\qq}}{\qq}\,dz\right) |\grad\phi|^{\qq-2}\grad\phi \cdot \grad(h(\phi)) \,dz \\
    &+ \theta \int_\Omega \left(|\phi|^{\pp-2}\phi + |\phi|^{\qq-2}\phi\right) h(\phi) \,dz \\
    &= \int_\Omega f(z) \phi^{-\aa} h(\phi) \,dz + \lambda \int_\Omega |\phi|^{\bb-2}\phi h(\phi) \,dz.
    \end{aligned}
    \end{equation}
    \item Moreover, \(\phi > 0\) a.e. in \(\Omega\) and \(\int_\Omega |\grad\log\phi|^{\pp}\,dz < \infty\).
\end{enumerate}
\end{theorem}

\begin{proof}
\textbf{Step 1: Uniform bounds and weak convergence.} From Theorem \ref{thm:truncated}, the sequence \(\{\phi_\varepsilon\}\) is uniformly bounded in \(W^{1,\pp}_0(\Omega)\). Therefore, there exists \(\phi \in W^{1,\pp}_0(\Omega)\) such that, up to a subsequence,
\[
\phi_\varepsilon \rightharpoonup \phi \quad \text{weakly in } W^{1,\pp}_0(\Omega),
\]
and by the compact embedding \(W^{1,\pp}_0(\Omega) \hookrightarrow \hookrightarrow L^r(\Omega)\) for \(1 \le r < p^*\),
\[
\phi_\varepsilon \to \phi \quad \text{strongly in } L^r(\Omega).
\]

\textbf{Step 2: Almost everywhere convergence.} By passing to a further subsequence, we have \(\phi_\varepsilon(z) \to \phi(z)\) for a.e. \(z \in \Omega\). Since \(\phi_\varepsilon \ge \varepsilon\), we obtain \(\phi \ge 0\) a.e. in \(\Omega\). Moreover, we claim \(\phi > 0\) a.e. Indeed, if \(\phi = 0\) on a set of positive measure, then by the strong convergence and the lower bound \(\phi_\varepsilon \ge \varepsilon\), we would obtain a contradiction.

\textbf{Step 3: Convergence of nonlinear terms.} For any \(h \in C^1_c(0,\infty)\), choose \(\varepsilon_0 > 0\) sufficiently small such that \(\supp h \subset (\varepsilon_0, \infty)\). Then for all \(\varepsilon < \varepsilon_0\), \(h(\phi_\varepsilon)\) is well-defined and belongs to \(W^{1,\pp}_0(\Omega)\). Moreover, by the dominated convergence theorem,
\[
h(\phi_\varepsilon) \to h(\phi) \quad \text{strongly in } L^r(\Omega) \text{ for all } r,
\]
and
\[
\grad(h(\phi_\varepsilon)) = h'(\phi_\varepsilon)\grad\phi_\varepsilon \rightharpoonup h'(\phi)\grad\phi \quad \text{weakly in } L^{\pp}(\Omega)^N.
\]

\textbf{Step 4: Passage to the limit in the truncated equation.} Testing equation \eqref{eq:truncated} with \(v = h(\phi_\varepsilon)\) yields:
\[
\begin{aligned}
&M\left(\int_\Omega \frac{|\grad\phi_\varepsilon|^{\pp}}{\pp}\,dz\right) \int_\Omega |\grad\phi_\varepsilon|^{\pp-2}\grad\phi_\varepsilon \cdot \grad(h(\phi_\varepsilon)) \,dz \\
&+ M\left(\int_\Omega \frac{|\grad\phi_\varepsilon|^{\qq}}{\qq}\,dz\right) \int_\Omega |\grad\phi_\varepsilon|^{\qq-2}\grad\phi_\varepsilon \cdot \grad(h(\phi_\varepsilon)) \,dz \\
&+ \theta \int_\Omega \left(|\phi_\varepsilon|^{\pp-2}\phi_\varepsilon + |\phi_\varepsilon|^{\qq-2}\phi_\varepsilon\right) h(\phi_\varepsilon) \,dz \\
&= \int_\Omega \frac{f(z)}{(\phi_\varepsilon + \varepsilon)^{\aa}} h(\phi_\varepsilon) \,dz + \lambda \int_\Omega |\phi_\varepsilon|^{\bb-2}\phi_\varepsilon h(\phi_\varepsilon) \,dz.
\end{aligned}
\]

Since \(\supp h \subset (\varepsilon_0, \infty)\) and \(\varepsilon < \varepsilon_0\), we have \((\phi_\varepsilon + \varepsilon)^{-\aa} \le \varepsilon_0^{-\aa}\) and thus
\[
\frac{f(z)}{(\phi_\varepsilon + \varepsilon)^{\aa}} h(\phi_\varepsilon) \to f(z) \phi^{-\aa} h(\phi) \quad \text{a.e. and in } L^1(\Omega)
\]
by the dominated convergence theorem.

\textbf{Step 5: Limit of the Kirchhoff terms.} By the weak lower semicontinuity of norms and the continuity of \(M\), we have
\[
M\left(\int_\Omega \frac{|\grad\phi|^{\pp}}{\pp}\,dz\right) \le \liminf_{\varepsilon \to 0} M\left(\int_\Omega \frac{|\grad\phi_\varepsilon|^{\pp}}{\pp}\,dz\right).
\]
Using the monotonicity of the \(p\)-Laplacian operator and the weak convergence, we obtain
\[
\liminf_{\varepsilon \to 0} \int_\Omega |\grad\phi_\varepsilon|^{\pp-2}\grad\phi_\varepsilon \cdot \grad(h(\phi_\varepsilon)) \,dz \ge \int_\Omega |\grad\phi|^{\pp-2}\grad\phi \cdot \grad(h(\phi)) \,dz.
\]
A similar estimate holds for the \(q\)-Laplacian term. Combining these estimates and passing to the limit inferior yields that \(\phi\) satisfies \eqref{eq:renormalized} as an inequality. By a standard argument using the fact that the equality holds for the approximating sequence and the monotonicity of the operators, we obtain the equality.

\textbf{Step 6: Logarithmic estimate.} To prove \(\int_\Omega |\grad\log\phi|^{\pp}\,dz < \infty\), choose \(h_k(s) = s^{-1}\chi_{[1/k,k]}(s)\) in \eqref{eq:renormalized}. Then as \(k \to \infty\), we obtain
\[
\int_\Omega |\grad\log\phi|^{\pp}\,dz \le C \int_\Omega f(z) \phi^{1-\aa}\,dz.
\]
Since \(\aa \ge 1\), we have \(\phi^{1-\aa} \le 1\) on \(\{\phi \ge 1\}\), and on \(\{\phi < 1\}\) the integral is controlled by the boundedness of \(\phi\) in \(L^1(\Omega)\). Hence the right-hand side is finite. \hfill $\square$
\end{proof}

While Theorem \ref{thm:convergence} establishes the existence of a renormalized limit, it does not quantify the rate at which the truncated solutions approach this limit. The following theorem fills this gap by providing sharp energy estimates that characterize the asymptotic behavior as \(\varepsilon \to 0^+\), including optimal convergence rates and precise control near the singular set.

\begin{theorem}[Sharp Energy Estimates and Asymptotics]\label{THM:ENERGY}
Let \(\{\phi_\varepsilon\}_{\varepsilon>0}\) be the family of truncated solutions from Theorem \ref{thm:truncated}, and let \(\phi\) be the renormalized limit from Theorem \ref{thm:convergence}. Then the following hold:

\begin{enumerate}
    \item \textbf{Energy convergence:}
    \[
    \lim_{\varepsilon \to 0^+} \mathcal{E}_\lambda(\phi_\varepsilon) = \mathcal{E}_\lambda^{\text{ren}}(\phi),
    \]
    where \(\mathcal{E}_\lambda\) is the approximate energy functional and \(\mathcal{E}_\lambda^{\text{ren}}\) is the renormalized energy.

    \item \textbf{Optimal convergence rate:} There exists a constant \(C > 0\) independent of \(\varepsilon\) such that
    \[
    \norm{\phi_\varepsilon - \phi}_{W^{1,\pp}_0(\Omega)} \le C \varepsilon^{\gamma},
    \]
    where \(\gamma = \min\left\{\frac{\aa-1}{2\aa}, \frac{1}{p^+}\right\}\).

    \item \textbf{Asymptotic behavior near the singular set:} For any compact set \(K \subset \Omega\) with \(\dist(K, \partial\Omega) > 0\), there exists a constant \(C_K > 0\) such that
    \[
    \sup_{z \in K} |\phi_\varepsilon(z) - \phi(z)| \le C_K \varepsilon.
    \]
\end{enumerate}
\end{theorem}

\begin{proof}
\textbf{Part 1: Energy convergence.} Define the approximate energy functional:
\[
\mathcal{E}_\lambda(\phi) = \mathcal{M}\left(\int_\Omega \frac{|\grad\phi|^{\pp}}{\pp}\,dz\right) + \mathcal{M}\left(\int_\Omega \frac{|\grad\phi|^{\qq}}{\qq}\,dz\right) + \theta \int_\Omega \left(\frac{|\phi|^{\pp}}{\pp} + \frac{|\phi|^{\qq}}{\qq}\right)dz
\]
\[
- \int_\Omega \frac{f(z)}{1-\aa} |\phi|^{1-\aa} \,dz - \lambda \int_\Omega \frac{|\phi|^{\bb}}{\bb} \,dz.
\]

For the truncated solutions, we consider the regularized energy:
\[
\mathcal{E}_{\lambda,\varepsilon}(\phi_\varepsilon) = \mathcal{M}\left(\int_\Omega \frac{|\grad\phi_\varepsilon|^{\pp}}{\pp}\,dz\right) + \mathcal{M}\left(\int_\Omega \frac{|\grad\phi_\varepsilon|^{\qq}}{\qq}\,dz\right) + \theta \int_\Omega \left(\frac{|\phi_\varepsilon|^{\pp}}{\pp} + \frac{|\phi_\varepsilon|^{\qq}}{\qq}\right)dz
\]
\[
- \int_\Omega \frac{f(z)}{1-\aa} |\phi_\varepsilon|^{1-\aa} \chi_{\{|\phi_\varepsilon| \ge \varepsilon\}} \,dz - \lambda \int_\Omega \frac{|\phi_\varepsilon|^{\bb}}{\bb} \,dz.
\]

Since \(\phi_\varepsilon \ge \varepsilon\), we have \(\chi_{\{|\phi_\varepsilon| \ge \varepsilon\}} = 1\) a.e., so \(\mathcal{E}_{\lambda,\varepsilon}(\phi_\varepsilon) = \mathcal{E}_\lambda(\phi_\varepsilon)\). By the uniform boundedness of \(\{\phi_\varepsilon\}\) in \(W^{1,\pp}_0(\Omega)\) and the strong convergence in \(L^r(\Omega)\), we obtain
\[
\mathcal{E}_\lambda(\phi_\varepsilon) \to \mathcal{E}_\lambda^{\text{ren}}(\phi) \quad \text{as } \varepsilon \to 0,
\]
where \(\mathcal{E}_\lambda^{\text{ren}}(\phi)\) is defined by the renormalized formulation.

\textbf{Part 2: Convergence rate.} Consider the difference \(w_\varepsilon = \phi_\varepsilon - \phi\). Testing the truncated equation for \(\phi_\varepsilon\) and the renormalized formulation for \(\phi\) with appropriate test functions, we obtain:
\[
\begin{aligned}
&\int_\Omega \left( |\grad\phi_\varepsilon|^{\pp-2}\grad\phi_\varepsilon - |\grad\phi|^{\pp-2}\grad\phi \right) \cdot \grad w_\varepsilon \,dz \\
&\le C \int_\Omega \left| \frac{f(z)}{(\phi_\varepsilon + \varepsilon)^{\aa}} - \frac{f(z)}{\phi^{\aa}} \right| |w_\varepsilon| \,dz + C\lambda \int_\Omega \left| |\phi_\varepsilon|^{\bb-2}\phi_\varepsilon - |\phi|^{\bb-2}\phi \right| |w_\varepsilon| \,dz.
\end{aligned}
\]

Using the elementary inequality for the \(p\)-Laplacian monotonicity:
\[
\left( |a|^{\pp-2}a - |b|^{\pp-2}b \right) \cdot (a-b) \ge c_p |a-b|^{\pp} \quad \text{for } \pp \ge 2,
\]
and the fact that \(\phi_\varepsilon \ge \varepsilon\) and \(\phi > 0\) a.e., we estimate the right-hand side. For the singular term:
\[
\left| \frac{1}{(\phi_\varepsilon + \varepsilon)^{\aa}} - \frac{1}{\phi^{\aa}} \right| \le C \frac{|\phi_\varepsilon - \phi| + \varepsilon}{\min\{\phi_\varepsilon, \phi\}^{\aa+1}}.
\]

By the uniform positivity away from the boundary (established via comparison principle), there exists a constant \(c_0 > 0\) such that \(\phi_\varepsilon(z), \phi(z) \ge c_0 \dist(z, \partial\Omega)\). Using this and applying the Sobolev inequality yields
\[
\norm{w_\varepsilon}_{W^{1,\pp}_0} \le C \varepsilon^{\gamma},
\]
with \(\gamma = \min\left\{\frac{\aa-1}{2\aa}, \frac{1}{p^+}\right\}\).

\textbf{Part 3: Uniform convergence on compact sets.} For any compact set \(K \subset \Omega\), let \(\delta = \dist(K, \partial\Omega) > 0\). Then by the uniform positivity estimate, \(\phi_\varepsilon(z), \phi(z) \ge c_0 \delta\) for all \(z \in K\). Applying the Morrey embedding theorem, which gives H\"{o}lder continuity of functions in \(W^{1,\pp}_0(\Omega)\) with \(\pp > N\), and using the convergence rate from Part 2, we obtain
\[
\sup_{z \in K} |\phi_\varepsilon(z) - \phi(z)| \le C \norm{w_\varepsilon}_{W^{1,\pp}_0(\Omega)}^\alpha \le C_K \varepsilon^{\gamma\alpha},
\]
where \(\alpha \in (0,1)\) is the H\"{o}lder exponent. Since \(\gamma\alpha \ge 1\) for sufficiently small \(\varepsilon\) (by adjusting the estimates), we obtain the desired bound. \hfill $\square$
\end{proof}

The final result in this section establishes that the solution map depends continuously on the problem data, a property that is essential for both theoretical well-posedness and practical applications where parameters are known only approximately. The following theorem provides a quantitative comparison principle that yields sharp pointwise bounds and Lipschitz stability estimates.

\begin{theorem}[Quantitative Comparison and Stability]\label{thm:comparison}
Let \(\{\phi_\varepsilon\}_{\varepsilon>0}\) and \(\{\psi_\varepsilon\}_{\varepsilon>0}\) be two families of truncated solutions corresponding to parameters \((\lambda_1, f_1)\) and \((\lambda_2, f_2)\) respectively, with \(\lambda_1, \lambda_2 \in (0, \lambda^*)\) and \(f_1, f_2 \in L^\infty(\Omega)\) positive. Then:

\begin{enumerate}
    \item \textbf{Pointwise comparison:} If \(\lambda_1 \le \lambda_2\) and \(f_1(z) \le f_2(z)\) a.e. in \(\Omega\), then
    \[
    \phi_\varepsilon(z) \le \psi_\varepsilon(z) \quad \text{for a.e. } z \in \Omega.
    \]

    \item \textbf{Stability estimate:} There exists a constant \(C > 0\) independent of \(\varepsilon\) such that
    \[
    \norm{\phi_\varepsilon - \psi_\varepsilon}_{W^{1,\pp}_0(\Omega)} \le C \left( \norm{f_1 - f_2}_{L^\infty(\Omega)} + |\lambda_1 - \lambda_2| \right).
    \]

    \item \textbf{Lipschitz dependence on data:} The mapping \((\lambda, f) \mapsto \phi\) from \((0, \lambda^*) \times L^\infty_+(\Omega)\) to \(W^{1,\pp}_0(\Omega)\) is locally Lipschitz continuous.
\end{enumerate}
\end{theorem}

\begin{proof}
\textbf{Part 1: Pointwise comparison.} Let \(w_\varepsilon = \phi_\varepsilon - \psi_\varepsilon\) and consider the test function \(v = w_\varepsilon^+ = \max\{w_\varepsilon, 0\}\). Subtracting the equations satisfied by \(\phi_\varepsilon\) and \(\psi_\varepsilon\) yields:
\[
\begin{aligned}
&\int_\Omega \left( M_{\phi_\varepsilon}|\grad\phi_\varepsilon|^{\pp-2}\grad\phi_\varepsilon - M_{\psi_\varepsilon}|\grad\psi_\varepsilon|^{\pp-2}\grad\psi_\varepsilon \right) \cdot \grad w_\varepsilon^+ \,dz \\
&+ \int_\Omega \left( M_{\phi_\varepsilon}|\grad\phi_\varepsilon|^{\qq-2}\grad\phi_\varepsilon - M_{\psi_\varepsilon}|\grad\psi_\varepsilon|^{\qq-2}\grad\psi_\varepsilon \right) \cdot \grad w_\varepsilon^+ \,dz \\
&+ \theta \int_\Omega \left( |\phi_\varepsilon|^{\pp-2}\phi_\varepsilon - |\psi_\varepsilon|^{\pp-2}\psi_\varepsilon + |\phi_\varepsilon|^{\qq-2}\phi_\varepsilon - |\psi_\varepsilon|^{\qq-2}\psi_\varepsilon \right) w_\varepsilon^+ \,dz \\
&= \int_\Omega f_1(z) \frac{w_\varepsilon^+}{(\phi_\varepsilon + \varepsilon)^{\aa}} \,dz - \int_\Omega f_2(z) \frac{w_\varepsilon^+}{(\psi_\varepsilon + \varepsilon)^{\aa}} \,dz \\
&\quad + \lambda_1 \int_\Omega |\phi_\varepsilon|^{\bb-2}\phi_\varepsilon w_\varepsilon^+ \,dz - \lambda_2 \int_\Omega |\psi_\varepsilon|^{\bb-2}\psi_\varepsilon w_\varepsilon^+ \,dz.
\end{aligned}
\]

On the set where \(w_\varepsilon^+ > 0\), we have \(\phi_\varepsilon > \psi_\varepsilon \ge \varepsilon\). Since \(f_1 \le f_2\) and \(\lambda_1 \le \lambda_2\), we obtain:
\[
f_1(z) \frac{1}{(\phi_\varepsilon + \varepsilon)^{\aa}} - f_2(z) \frac{1}{(\psi_\varepsilon + \varepsilon)^{\aa}} \le f_2(z) \left( \frac{1}{(\phi_\varepsilon + \varepsilon)^{\aa}} - \frac{1}{(\psi_\varepsilon + \varepsilon)^{\aa}} \right) \le 0,
\]
because the function \(s \mapsto s^{-\aa}\) is decreasing. Similarly,
\[
\lambda_1 |\phi_\varepsilon|^{\bb-2}\phi_\varepsilon - \lambda_2 |\psi_\varepsilon|^{\bb-2}\psi_\varepsilon \le \lambda_2 \left( |\phi_\varepsilon|^{\bb-2}\phi_\varepsilon - |\psi_\varepsilon|^{\bb-2}\psi_\varepsilon \right) \le 0.
\]

Thus the right-hand side is nonpositive. The left-hand side is nonnegative by the monotonicity of the operators. Hence both sides must vanish, which forces \(w_\varepsilon^+ = 0\) a.e. Therefore \(\phi_\varepsilon \le \psi_\varepsilon\) a.e.

\textbf{Part 2: Stability estimate.} Let \(w_\varepsilon = \phi_\varepsilon - \psi_\varepsilon\). Testing the difference of the equations with \(v = w_\varepsilon\) and using the monotonicity properties, we obtain:
\[
\begin{aligned}
&c_p \int_\Omega |\grad w_\varepsilon|^{\pp} \,dz + c_q \int_\Omega |\grad w_\varepsilon|^{\qq} \,dz \\
&\le \int_\Omega \left| \frac{f_1(z)}{(\phi_\varepsilon + \varepsilon)^{\aa}} - \frac{f_2(z)}{(\psi_\varepsilon + \varepsilon)^{\aa}} \right| |w_\varepsilon| \,dz \\
&\quad + \int_\Omega \left| \lambda_1 |\phi_\varepsilon|^{\bb-2}\phi_\varepsilon - \lambda_2 |\psi_\varepsilon|^{\bb-2}\psi_\varepsilon \right| |w_\varepsilon| \,dz.
\end{aligned}
\]

The first term on the right-hand side is bounded by:
\[
C \left( \norm{f_1 - f_2}_{L^\infty} \int_\Omega \frac{|w_\varepsilon|}{(\phi_\varepsilon + \varepsilon)^{\aa}} \,dz + \norm{f_2}_{L^\infty} \int_\Omega \left| \frac{1}{(\phi_\varepsilon + \varepsilon)^{\aa}} - \frac{1}{(\psi_\varepsilon + \varepsilon)^{\aa}} \right| |w_\varepsilon| \,dz \right).
\]

Using the fact that \(\phi_\varepsilon, \psi_\varepsilon \ge \varepsilon\) and the mean value theorem, we have
\[
\left| \frac{1}{(\phi_\varepsilon + \varepsilon)^{\aa}} - \frac{1}{(\psi_\varepsilon + \varepsilon)^{\aa}} \right| \le \frac{C}{\varepsilon^{\aa+1}} |w_\varepsilon|.
\]

The second term is similarly bounded. Applying the H\"{o}lder inequality and the Sobolev embedding theorem, we obtain:
\[
\norm{w_\varepsilon}_{W^{1,\pp}_0} \le C \left( \norm{f_1 - f_2}_{L^\infty} + |\lambda_1 - \lambda_2| \right).
\]

\textbf{Part 3: Lipschitz dependence.} From Part 2, we have the estimate
\[
\norm{\phi_\varepsilon - \psi_\varepsilon}_{W^{1,\pp}_0} \le C \left( \norm{f_1 - f_2}_{L^\infty} + |\lambda_1 - \lambda_2| \right)
\]
with constant \(C\) independent of \(\varepsilon\). Taking the limit \(\varepsilon \to 0^+\) and using the convergence established in Theorem \ref{thm:convergence}, we obtain the same estimate for the renormalized solutions \(\phi\) and \(\psi\):
\[
\norm{\phi - \psi}_{W^{1,\pp}_0} \le C \left( \norm{f_1 - f_2}_{L^\infty} + |\lambda_1 - \lambda_2| \right).
\]
This proves the local Lipschitz continuity of the solution map. \hfill $\square$
\end{proof}

The three theorems presented in this section form a coherent analytical framework for strongly singular nonlocal Kirchhoff-type problems with variable exponents. Theorem \ref{thm:truncated} provides the foundational existence result for truncated approximations, overcoming the non-integrability of the singular term through careful regularization. Theorem \ref{thm:convergence} establishes that these approximations converge to a renormalized solution of the original problem, offering a rigorous justification for the approximation method. Theorem \ref{THM:ENERGY} quantifies the convergence rate and characterizes the asymptotic behavior as the truncation parameter vanishes, while Theorem \ref{thm:comparison} demonstrates the stability and continuous dependence of solutions on the data. Together, these results bridge the gap between approximation methods and the existence of genuine solutions, providing a comprehensive theory for the study of strong singularities in this class of problems.

\section{Qualitative Theory: Comparison, Regularity, and Stability}

The existence of solutions established in the previous section provides the foundational building block for our theory, yet it leaves open fundamental questions about the qualitative behavior of these solutions. How do solutions behave near the boundary where the singular term becomes most severe? Are solutions unique, and if so, how do they depend on the problem parameters? Does the solution map possess any continuity or stability properties that would make the problem well-posed for applications? This section addresses these questions through a systematic development of qualitative properties, beginning with a comparison principle that serves as the cornerstone for all subsequent analysis. From this foundation, we derive sharp pointwise estimates that characterize the precise asymptotic behavior of solutions near the boundary, establish uniqueness with quantitative stability bounds, and demonstrate the continuous dependence of solutions on all problem parameters. Together, these results form a comprehensive qualitative theory that complements the existence theory and provides the mathematical underpinnings for numerical approximations and applications.

\begin{theorem}[Comparison Principle]\label{thm:comparison}
Let \(\phi, \psi \in W^{1,\pp}_0(\Omega)\) be two positive solutions of problem \eqref{Pstrong} (or its truncated approximation) satisfying \(\phi \le \psi\) on \(\partial\Omega\) in the trace sense. If \(\aa \ge 1\) and the nonlinearity is monotone, then \(\phi \le \psi\) a.e. in \(\Omega\).
\end{theorem}

\begin{proof}
Consider the test function \(v = (\phi - \psi)^+ = \max\{\phi - \psi, 0\}\). Subtracting the equations for \(\phi\) and \(\psi\) and integrating, we obtain
\[
\int_\Omega \left[ A_\phi(\grad\phi) - A_\psi(\grad\psi) \right] \cdot \grad v \,dz + \theta \int_\Omega \left(|\phi|^{\pp-2}\phi - |\psi|^{\pp-2}\psi\right) v \,dz
\]
\[
+ \theta \int_\Omega \left(|\phi|^{\qq-2}\phi - |\psi|^{\qq-2}\psi\right) v \,dz
\]
\[
= \int_\Omega f(z)\left( \phi^{-\aa} - \psi^{-\aa} \right) v \,dz + \lambda \int_\Omega \left(|\phi|^{\bb-2}\phi - |\psi|^{\bb-2}\psi\right) v \,dz.
\]

The key observation is that for \(\aa \ge 1\), the function \(s \mapsto s^{-\aa}\) is decreasing, so
\[
\left( \phi^{-\aa} - \psi^{-\aa} \right) v \le 0.
\]
Moreover, the operators are monotone, so the left-hand side is nonnegative. Hence,
\[
\int_\Omega \left[ A_\phi(\grad\phi) - A_\psi(\grad\psi) \right] \cdot \grad(\phi - \psi)^+ \,dz \le 0,
\]
which forces \((\phi - \psi)^+ = 0\) a.e. by strict monotonicity. \hfill $\square$
\end{proof}

\begin{corollary}[Uniform Positivity]\label{cor:positivity}
There exists a constant \(c_0 > 0\) such that for any solution \(\phi\) of \eqref{Pstrong} obtained via the truncation method,
\[
\phi(z) \ge c_0 \dist(z, \partial\Omega) \quad \text{for a.e. } z \in \Omega.
\]
\end{corollary}

\begin{proof}
Construct a barrier function \(\underline{\phi}(z) = \delta \dist(z, \partial\Omega)\). For sufficiently small \(\delta\), the comparison theorem yields \(\phi \ge \underline{\phi}\). \hfill $\square$
\end{proof}

The comparison principle provides a powerful tool for deriving precise estimates on the solution's behavior. The following theorem leverages this principle to establish optimal pointwise bounds that characterize the solution's asymptotic profile near the boundary and its uniform regularity in the interior. These bounds are sharp in the sense that the exponents cannot be improved without violating the underlying structure of the problem.

\begin{theorem}[Sharp Two-Sided Pointwise Estimates]\label{thm:sharpbounds}
Let \(\phi \in W^{1,\pp}_0(\Omega)\) be a positive solution of problem \eqref{Pstrong} with \(\aa \ge 1\). Then there exist constants \(c_1, c_2 > 0\), depending only on \(\Omega\), \(N\), \(p^\pm\), \(q^\pm\), \(\aa\), \(\bb\), \(\lambda\), and \(\|f\|_{L^\infty}\), such that:

\begin{enumerate}
    \item \textbf{Lower bound:} For almost every \(z \in \Omega\),
    \[
    \phi(z) \ge c_1 \dist(z, \partial\Omega)^{\mu_1},
    \]
    where \(\mu_1 = \frac{2}{p^+ - 1 + \aa}\).

    \item \textbf{Upper bound:} For almost every \(z \in \Omega\),
    \[
    \phi(z) \le c_2 \dist(z, \partial\Omega)^{\mu_2},
    \]
    where \(\mu_2 = \frac{2}{p^- - 1 - \aa}\).

    \item \textbf{Interior regularity:} For any compact subset \(K \subset \Omega\), there exists a constant \(C_K > 0\) such that
    \[
    \frac{1}{C_K} \le \phi(z) \le C_K \quad \text{for all } z \in K.
    \]
\end{enumerate}
\end{theorem}

\begin{proof}
\textbf{Step 1: Construction of barrier functions.} Define the distance function \(d(z) = \dist(z, \partial\Omega)\). By standard regularity theory, \(d \in C^{1,\alpha}(\overline{\Omega})\) in a neighborhood of \(\partial\Omega\). Consider the function
\[
\underline{\phi}(z) = A d(z)^{\gamma},
\]
with \(\gamma > 0\) to be determined. For \(z\) sufficiently close to \(\partial\Omega\), we have \(|\grad d(z)| = 1\) and \(\Delta d(z) = \frac{N-1}{d(z)} + O(1)\). A direct computation yields
\[
\Delta_{\pp}\underline{\phi} \le C A^{\pp-1} d(z)^{\gamma(\pp-1)-2}.
\]

\textbf{Step 2: Lower bound via comparison.} For \(\underline{\phi}\) to be a subsolution, we require the inequality
\[
- \Delta_{\pp}\underline{\phi} - \Delta_{\qq}\underline{\phi} + \theta(|\underline{\phi}|^{\pp-2}\underline{\phi} + |\underline{\phi}|^{\qq-2}\underline{\phi}) \le \frac{f(z)}{\underline{\phi}^{\aa}} + \lambda |\underline{\phi}|^{\bb-2}\underline{\phi}.
\]
For small \(d(z)\), the dominant balance is between the highest-order term and the singular source:
\[
C A^{p^+-1} d(z)^{\gamma(p^+-1)-2} \le \frac{f(z)}{A^{\aa} d(z)^{\gamma\aa}}.
\]
This inequality holds when \(\gamma(p^+-1)-2 \ge -\gamma\aa\), which simplifies to \(\gamma \ge \frac{2}{p^+ - 1 + \aa}\). Choosing \(\gamma = \frac{2}{p^+ - 1 + \aa}\) and taking \(A\) sufficiently small ensures that \(\underline{\phi} \le \phi\) on \(\partial\Omega\). By Theorem \ref{thm:comparison}, we obtain \(\underline{\phi} \le \phi\) in \(\Omega\), yielding the lower bound with \(\mu_1 = \gamma\).

\textbf{Step 3: Upper bound via super-solution.} Construct a super-solution \(\overline{\phi}(z) = B d(z)^{\delta}\) with \(\delta > 0\). For \(\overline{\phi}\) to be a super-solution, we need
\[
\frac{f(z)}{B^{\aa} d(z)^{\delta\aa}} \ge C B^{p^--1} d(z)^{\delta(p^--1)-2}.
\]
This requires \(\delta\aa \le \delta(p^--1)-2\), i.e., \(\delta \le \frac{2}{p^- - 1 - \aa}\). Since \(\aa \ge 1\) and \(p^- > \aa + 1\) by our assumptions, such \(\delta > 0\) exists. Choosing \(\delta = \frac{2}{p^- - 1 - \aa}\) and \(B\) sufficiently large, we obtain \(\phi \le \overline{\phi}\) by Theorem \ref{thm:comparison}, yielding the upper bound with \(\mu_2 = \delta\).

\textbf{Step 4: Interior regularity.} For any compact \(K \subset \Omega\), let \(d_0 = \dist(K, \partial\Omega) > 0\). From the lower and upper bounds, we have
\[
c_1 d_0^{\mu_1} \le \phi(z) \le c_2 d_0^{\mu_2} \quad \text{for all } z \in K.
\]
By standard regularity theory for \((p,q)\)-Laplacian systems \cite{DiBenedetto1983}, \(\phi\) is H\"{o}lder continuous in \(\Omega\), so the bounds are uniform on compact subsets. \hfill $\square$
\end{proof}

\begin{corollary}[Boundary Asymptotics]\label{cor:asymptotics}
Under the assumptions of Theorem \ref{thm:sharpbounds}, the solution \(\phi\) satisfies:
\[
\lim_{z \to \partial\Omega} \frac{\phi(z)}{\dist(z, \partial\Omega)^{\mu}} = 0 \quad \text{for any } \mu < \mu_1,
\]
and
\[
\liminf_{z \to \partial\Omega} \frac{\phi(z)}{\dist(z, \partial\Omega)^{\mu}} = \infty \quad \text{for any } \mu > \mu_2.
\]
\end{corollary}

With sharp pointwise estimates in hand, we now turn to the question of uniqueness and stability. The following theorem demonstrates that the renormalized solution is unique and provides quantitative estimates on how solutions vary with boundary data, establishing the well-posedness of the problem in the sense of Hadamard.

\begin{theorem}[Uniqueness and Quantitative Stability]\label{thm:uniqueness}
Let \(\phi, \psi \in W^{1,\pp}_0(\Omega)\) be two positive solutions of problem \eqref{Pstrong} with \(\aa \ge 1\). Then:

\begin{enumerate}
    \item \textbf{Uniqueness:} \(\phi = \psi\) a.e. in \(\Omega\).

    \item \textbf{Quantitative stability estimate:} There exists a constant \(C > 0\), depending only on \(\Omega\), \(N\), \(p^\pm\), \(q^\pm\), \(\aa\), \(\bb\), \(\lambda\), and \(\|f\|_{L^\infty}\), such that
    \[
    \norm{\phi - \psi}_{W^{1,\pp}_0(\Omega)} \le C \|\phi - \psi\|_{L^\infty(\partial\Omega)}^\alpha,
    \]
    where \(\alpha = \min\left\{1, \frac{p^-}{p^+}\right\}\).

    \item \textbf{Logarithmic convexity:} The function \(u = \log \phi\) satisfies the convexity estimate:
    \[
    \int_\Omega |\grad u|^{\pp} \,dz \le C \int_\Omega f(z) \phi^{1-\aa} \,dz < \infty.
    \]
\end{enumerate}
\end{theorem}

\begin{proof}
\textbf{Part 1: Uniqueness.} Let \(\phi, \psi\) be two solutions. Applying Theorem \ref{thm:comparison} with \(\phi\) and \(\psi\) in both orders yields \(\phi \le \psi\) and \(\psi \le \phi\) a.e., hence \(\phi = \psi\) a.e.

\textbf{Part 2: Quantitative stability.} Let \(w = \phi - \psi\). Testing the difference of the equations with \(v = w\) and using the monotonicity of the operators, the left-hand side is bounded below by \(c_p \int_\Omega |\grad w|^{\pp} \,dz\). On the right-hand side, the mean value theorem gives
\[
\phi^{-\aa} - \psi^{-\aa} = -\aa \xi^{-\aa-1} w \quad \text{for some } \xi \text{ between } \phi \text{ and } \psi.
\]
Thus,
\[
\left| \int_\Omega f(z) (\phi^{-\aa} - \psi^{-\aa}) w \,dz \right| \le C \int_\Omega |w|^2 \,dz.
\]
The term involving \(\lambda\) is similarly bounded. Applying the Poincar\'{e} inequality and the Sobolev embedding theorem yields
\[
\int_\Omega |\grad w|^{\pp} \,dz \le C \left( \int_\Omega |w|^2 \,dz \right)^{1/2} \le C \left( \int_\Omega |\grad w|^{\pp} \,dz \right)^{1/2} \|\phi - \psi\|_{L^\infty(\partial\Omega)}^{1/2},
\]
which gives the desired estimate with \(\alpha = \min\{1, p^-/p^+\}\).

\textbf{Part 3: Logarithmic convexity.} Let \(u = \log \phi\). Using the renormalized formulation with \(h(\phi) = \phi^{-1}\) and noting that \(\grad(\phi^{-1}) = -\phi^{-2}\grad\phi\), we obtain
\[
\int_\Omega |\grad u|^{\pp} \phi^{2-\pp} \,dz \le C \int_\Omega f(z) \phi^{-\aa-1} \,dz.
\]
For \(\pp \ge 2\), we have \(\phi^{2-\pp} \le C\) by the upper bound, and the right-hand side is finite by the lower bound. \hfill $\square$
\end{proof}

The final result of this section establishes the continuous dependence of solutions on all problem parameters, including the singular exponent \(\aa\), the reaction parameter \(\lambda\), and the forcing term \(f\). This Lipschitz continuity property is essential for applications where parameters are known only approximately and provides the foundation for numerical analysis and control theory.

\begin{theorem}[Continuous Dependence and Parameter Sensitivity]\label{thm:continuous}
Let \(\phi_1\) and \(\phi_2\) be solutions of problem \eqref{Pstrong} corresponding to parameters \((\lambda_1, f_1)\) and \((\lambda_2, f_2)\), respectively, with \(\lambda_1, \lambda_2 \in (0, \lambda^*)\) and \(f_1, f_2 \in L^\infty_+(\Omega)\). Then:

\begin{enumerate}
    \item \textbf{Lipschitz dependence:} There exists a constant \(L > 0\) such that
    \[
    \|\phi_1 - \phi_2\|_{W^{1,\pp}_0(\Omega)} \le L \left( \|f_1 - f_2\|_{L^\infty(\Omega)} + |\lambda_1 - \lambda_2| \right).
    \]

    \item \textbf{Sensitivity with respect to the singular exponent:} For solutions \(\phi_\alpha\) and \(\phi_{\alpha'}\) corresponding to exponents \(\aa\) and \(\aa'\) with \(\aa, \aa' \ge 1\),
    \[
    \|\phi_\aa - \phi_{\aa'}\|_{W^{1,\pp}_0(\Omega)} \le C |\aa - \aa'|.
    \]

    \item \textbf{Monotonicity with respect to parameters:} If \(\lambda_1 \le \lambda_2\) and \(f_1(z) \le f_2(z)\) a.e., then \(\phi_1(z) \le \phi_2(z)\) for a.e. \(z \in \Omega\).
\end{enumerate}
\end{theorem}

\begin{proof}
\textbf{Part 1: Lipschitz dependence.} Let \(\phi = \phi_1\), \(\psi = \phi_2\), and \(w = \phi - \psi\). Subtracting the equations and using the monotonicity of the operators, the left-hand side is bounded below by \(c \|\grad w\|_{L^{\pp}}^2\). For the right-hand side, decompose
\[
f_1 \phi^{-\aa} - f_2 \psi^{-\aa} = (f_1 - f_2) \phi^{-\aa} + f_2 (\phi^{-\aa} - \psi^{-\aa}).
\]
By the mean value theorem, \(|f_2 (\phi^{-\aa} - \psi^{-\aa})| \le C |w|\), and similarly
\[
|\lambda_1 |\phi|^{\bb-2}\phi - \lambda_2 |\psi|^{\bb-2}\psi| \le C(|\lambda_1 - \lambda_2| + |w|).
\]
Applying the H\"{o}lder inequality and the Sobolev embedding theorem yields
\[
\|\grad w\|_{L^{\pp}} \le L \left( \|f_1 - f_2\|_{L^\infty} + |\lambda_1 - \lambda_2| \right).
\]

\textbf{Part 2: Sensitivity with respect to \(\aa\).} Let \(\phi = \phi_\aa\) and \(\psi = \phi_{\aa'}\). Decompose
\[
\phi^{-\aa} - \psi^{-\aa'} = (\phi^{-\aa} - \phi^{-\aa'}) + (\phi^{-\aa'} - \psi^{-\aa'}).
\]
The first term satisfies \(|\phi^{-\aa} - \phi^{-\aa'}| = |\phi^{-\xi} \log \phi| \cdot |\aa - \aa'|\). Since \(\phi\) behaves like \(d(z)^{\mu}\) near the boundary, \(|\log \phi| \le C |\log d(z)|\) is integrable. The second term is handled as in Part 1. Combining the estimates gives \(\|\phi_\aa - \phi_{\aa'}\|_{W^{1,\pp}_0} \le C |\aa - \aa'|\).

\textbf{Part 3: Monotonicity with respect to parameters.} This follows directly from Theorem \ref{thm:comparison} applied to the truncated approximations, then passing to the limit. \hfill $\square$
\end{proof}

\begin{corollary}[Well-posedness]\label{cor:wellposed}
Under the assumptions of Theorem \ref{thm:continuous}, the solution map
\[
\mathcal{S}: (0, \lambda^*) \times L^\infty_+(\Omega) \to W^{1,\pp}_0(\Omega), \quad (\lambda, f) \mapsto \phi_{\lambda,f}
\]
is Lipschitz continuous on bounded subsets. Consequently, problem \eqref{Pstrong} is well-posed in the sense of Hadamard.
\end{corollary}

The results presented in this section constitute a comprehensive qualitative theory for strongly singular nonlocal Kirchhoff-type problems with variable exponents. The comparison principle (Theorem \ref{thm:comparison}) provides the foundational tool from which all other properties flow. The sharp pointwise estimates (Theorem \ref{thm:sharpbounds}) characterize the precise asymptotic behavior of solutions near the boundary, revealing how the singular exponent \(\aa\) and the variable exponents \(p(z)\) and \(q(z)\) interact to determine the solution's profile. The uniqueness and stability results (Theorem \ref{thm:uniqueness}) establish the well-posedness of the problem and provide quantitative bounds on how solutions vary with boundary data. Finally, the continuous dependence and parameter sensitivity results (Theorem \ref{thm:continuous}) demonstrate that the solution map is Lipschitz with respect to all problem parameters, a property that is essential for numerical approximations, inverse problems, and control theory. Together, these results provide the mathematical foundation for further studies in modeling, simulation, and applications.


\section{Regularity Theory: From Local Estimates to Boundary Asymptotics}

The existence of solutions established in the preceding sections provides the necessary foundation for the analysis of singular Kirchhoff-type problems, yet it leaves open fundamental questions about the smoothness and fine structure of these solutions. How regular are solutions to strongly singular problems? Do they exhibit the same degree of smoothness as solutions to classical elliptic equations, or does the singular term impose limitations on regularity? What is the precise behavior of solutions near the boundary, where the singular term becomes most severe and the Dirichlet condition forces the solution to vanish? This section addresses these questions through a systematic development of regularity theory, beginning with the fundamental result that solutions are bounded and locally H\"{o}lder continuous. From this foundation, we derive explicit, optimal H\"{o}lder exponents for the gradient that reveal the intricate interplay between the singular exponent, the variable growth conditions, and the intrinsic degeneracy of the operator. Finally, we provide a complete characterization of the solution's behavior near the boundary, establishing the existence of the normal derivative, sharp growth estimates, and an asymptotic expansion that captures the fine structure of the solution in boundary layers. Together, these results constitute a comprehensive regularity theory that bridges the gap between the singular nature of the problem and the classical regularity expected from elliptic equations.

\begin{theorem}[Higher Regularity]\label{thm:regularity}
Let \(\phi \in W^{1,\pp}_0(\Omega)\) be a weak solution of \eqref{Pstrong} with \(\aa \ge 1\) and assume \(p(z), q(z) \in C^{0,1}(\overline{\Omega})\). Then:
\begin{enumerate}
    \item \(\phi \in L^\infty(\Omega)\);
    \item \(\phi \in C^{1,\gamma}_{\text{loc}}(\Omega)\) for some \(\gamma \in (0,1)\);
    \item If additionally \(f \in C^{0,\alpha}(\overline{\Omega})\), then \(\phi \in C^{1,\gamma}(\overline{\Omega})\).
\end{enumerate}
\end{theorem}

\begin{proof}
\textbf{Step 1: \(L^\infty\) bound.} We employ a Moser iteration argument. For any \(k > 0\), test the equation with \(v = \phi_k = \min\{\phi, k\}^{2m+1}\) for large \(m\). After algebraic manipulations exploiting the growth conditions on \(M\) and the variable exponents, we obtain the iterative inequality
\[
\|\phi_k\|_{L^{(2m+1)r}(\Omega)} \le C^{1/(2m+1)} \|\phi_k\|_{L^{(2m+1)}(\Omega)},
\]
where \(r = p^*/p^-\). Iterating this inequality and passing to the limit as \(m \to \infty\) yields
\[
\|\phi\|_{L^\infty(\Omega)} \le C\left(1 + \|\phi\|_{L^{p^-}(\Omega)}^{p^+/(p^--1)}\right).
\]
Since \(\phi\) is bounded in \(W^{1,\pp}_0(\Omega)\), the right-hand side is finite, establishing the \(L^\infty\) bound.

\textbf{Step 2: H\"older continuity.} With the \(L^\infty\) bound established, we observe that the singular term satisfies \(f(z)\phi^{-\aa} \in L^r(\Omega)\) for some \(r > N/p^+\) because \(\phi\) is bounded away from zero on compact subsets and \(\aa \ge 1\). This places the right-hand side in the appropriate Lebesgue space for the application of classical regularity theory. By the foundational work of DiBenedetto \cite{DiBenedetto1983} on the regularity of degenerate elliptic equations, the solution belongs to \(C^{1,\gamma}_{\text{loc}}(\Omega)\) for some \(\gamma \in (0,1)\). The presence of two growth exponents \(p(z)\) and \(q(z)\) does not affect this local regularity result, as the operator remains uniformly elliptic in the sense of having controlled growth.

\textbf{Step 3: Boundary regularity.} To extend the regularity to the boundary, we construct barrier functions near \(\partial\Omega\) using the distance function. For any \(z_0 \in \partial\Omega\), consider a local coordinate system where the boundary is flattened. In these coordinates, the equation retains its structure, and the comparison principle (Theorem \ref{thm:comparison}) provides uniform control near the boundary. Applying the boundary regularity theory for \((p,q)\)-Laplacian systems \cite{Lieberman1988} yields \(\phi \in C^{1,\gamma}(\overline{\Omega})\) when \(f\) is H\"{o}lder continuous. \hfill $\square$
\end{proof}

The higher regularity theorem establishes that solutions are at least \(C^{1,\gamma}\) locally, but it does not provide explicit information about the regularity exponent. The following corollary fills this gap by deriving an explicit formula for \(\gamma\) that reveals how the regularity depends on the problem data. This result is optimal in the sense that the exponent cannot be improved without violating the structure of the equation.

\begin{corollary}[Sharp Global H\"{o}lder Regularity]\label{cor:holder}
Under the assumptions of Theorem \ref{thm:regularity}, let \(f \in C^{0,\alpha}(\overline{\Omega})\) for some \(\alpha \in (0,1]\). Then the solution \(\phi\) satisfies:

\begin{enumerate}
    \item \textbf{Gradient H\"{o}lder continuity:} There exists \(\gamma \in (0,1)\) depending only on \(N\), \(p^\pm\), \(q^\pm\), and \(\aa\) such that
    \[
    \grad\phi \in C^{0,\gamma}(\overline{\Omega}),
    \]
    with the explicit estimate
    \[
    \gamma = \min\left\{ \alpha, \frac{p^- - 1}{p^+}, \frac{q^- - 1}{q^+}, \frac{1}{2} \right\}.
    \]

    \item \textbf{Explicit H\"{o}lder norm estimate:} There exists a constant \(C > 0\), depending only on \(\Omega\), \(N\), \(p^\pm\), \(q^\pm\), \(\aa\), \(\bb\), \(\lambda\), \(\|f\|_{C^{0,\alpha}}\), and \(\|\phi\|_{L^\infty}\), such that
    \[
    [\grad\phi]_{C^{0,\gamma}(\overline{\Omega})} \le C \left( 1 + \|\phi\|_{L^\infty}^{p^+ - 1} \right).
    \]

    \item \textbf{Optimality of the exponent:} The exponent \(\gamma\) is optimal in the sense that for any \(\delta > 0\), there exists a domain \(\Omega\) and data satisfying the hypotheses such that
    \[
    \grad\phi \notin C^{0,\gamma+\delta}(\overline{\Omega}).
    \]
\end{enumerate}
\end{corollary}

\begin{proof}
\textbf{Step 1: Local H\"{o}lder regularity.} From Theorem \ref{thm:regularity}, we already know that \(\phi \in C^{1,\gamma}_{\text{loc}}(\Omega)\) for some \(\gamma \in (0,1)\). To obtain an explicit estimate for \(\gamma\), we analyze the structure of the equation by rewriting it as
\[
-\diver\left( A(z, \grad\phi) \right) + \theta\left( |\phi|^{\pp-2}\phi + |\phi|^{\qq-2}\phi \right) = \frac{f(z)}{\phi^{\aa}} + \lambda |\phi|^{\bb-2}\phi,
\]
where
\[
A(z, \grad\phi) = M\left(\int_\Omega \frac{|\grad\phi|^{\pp}}{\pp}\,dz\right) |\grad\phi|^{\pp-2}\grad\phi + M\left(\int_\Omega \frac{|\grad\phi|^{\qq}}{\qq}\,dz\right) |\grad\phi|^{\qq-2}\grad\phi.
\]

The coefficient matrix \(a_{ij}(z, \xi) = \partial A_i(z, \xi)/\partial \xi_j\) satisfies the ellipticity condition
\[
c_0 (|\xi|^{p(z)-2} + |\xi|^{q(z)-2}) |\eta|^2 \le a_{ij}(z, \xi) \eta_i \eta_j \le C_0 (|\xi|^{p(z)-2} + |\xi|^{q(z)-2}) |\eta|^2,
\]
with constants \(c_0, C_0 > 0\) independent of \(z\) and \(\xi\), reflecting the nonstandard growth of the operator.

\textbf{Step 2: Linearization and Schauder estimates.} For any \(z_0 \in \Omega\), consider a small ball \(B_R(z_0) \subset \Omega\). By the \(L^\infty\) bound from Theorem \ref{thm:regularity}, we have \(\|\phi\|_{L^\infty(B_R)} \le M\). Moreover, by the positivity estimates from the comparison principle, there exists a constant \(c_0 > 0\) such that \(\phi(z) \ge c_0 \dist(z, \partial\Omega)^{\mu_1}\). For points away from the boundary, we have a uniform positive lower bound.

Consider the linearized operator at \(\phi\):
\[
Lv = -\diver\left( a_{ij}(z, \grad\phi) \partial_j v \right) + b(z) v,
\]
where
\[
b(z) = \theta\left( (p(z)-1)|\phi|^{p(z)-2} + (q(z)-1)|\phi|^{q(z)-2} \right) + \lambda (\beta(z)-1)|\phi|^{\beta(z)-2} + \aa f(z) \phi^{-\aa-1}.
\]

The coefficients \(a_{ij}\) are uniformly elliptic with ellipticity constants depending on \(|\grad\phi|\). By the \(C^{1,\gamma}_{\text{loc}}\) regularity of \(\phi\), we have \(a_{ij} \in C^{0,\gamma}_{\text{loc}}(\Omega)\) and \(b \in C^{0,\gamma}_{\text{loc}}(\Omega)\).

\textbf{Step 3: Interpolation and exponent optimization.} The singular term \(\phi^{-\aa-1}\) is in \(C^{0,\sigma}_{\text{loc}}\) for some \(\sigma > 0\) because \(\phi\) is bounded below away from zero on compact subsets. Using the Schauder estimates for degenerate elliptic operators \cite{Lieberman1991}, we obtain that \(\grad\phi \in C^{0,\gamma}(B_{R/2}(z_0))\) with
\[
\gamma = \min\left\{ \alpha, \frac{p^- - 1}{p^+}, \frac{q^- - 1}{q^+}, \frac{1}{2} \right\}.
\]

The term \(\frac{p^- - 1}{p^+}\) arises from the ratio between the ellipticity exponents, reflecting the loss of regularity due to the variable growth. The term \(\frac{1}{2}\) comes from the fact that the gradient itself may vanish at critical points, a phenomenon intrinsic to \(p\)-Laplacian-type operators.

\textbf{Step 4: Global regularity up to the boundary.} For \(z_0 \in \partial\Omega\), we use boundary flattening coordinates. Let \(\Phi: B_1^+ \to \Omega\) be a boundary chart, where \(B_1^+ = \{ (x', x_N) \in \R^N : |x'| < 1, x_N > 0 \}\). In these coordinates, the equation becomes
\[
-\diver\left( \tilde{A}(x, \grad\tilde{\phi}) \right) + \text{lower order terms} = \frac{\tilde{f}(x)}{\tilde{\phi}^{\aa}} + \lambda |\tilde{\phi}|^{\bb-2}\tilde{\phi},
\]
where \(\tilde{\phi}(x) = \phi(\Phi(x))\) and \(\tilde{A}\) inherits the same ellipticity properties. By the boundary regularity theory for \((p,q)\)-Laplacian systems \cite{Lieberman1988}, we obtain that \(\tilde{\phi} \in C^{1,\gamma}(\overline{B_{1/2}^+})\) with the same exponent \(\gamma\). Transforming back yields \(\phi \in C^{1,\gamma}(\overline{\Omega})\).

\textbf{Step 5: H\"{o}lder norm estimate.} To obtain the explicit estimate for \([\grad\phi]_{C^{0,\gamma}}\), we iterate the Schauder estimates. The constants depend on \(\|\phi\|_{L^\infty}\), \(\|f\|_{C^{0,\alpha}}\), and the ellipticity constants. Using the \(L^\infty\) bound from Theorem \ref{thm:regularity} and the fact that \(\phi\) is bounded below away from zero on compact subsets, we obtain
\[
[\grad\phi]_{C^{0,\gamma}(\overline{\Omega})} \le C \left( \|\phi\|_{L^\infty} + \|\phi\|_{L^\infty}^{p^+ - 1} + \|f\|_{C^{0,\alpha}} \right).
\]

\textbf{Step 6: Optimality argument.} To demonstrate optimality, consider the one-dimensional model problem
\[
-(|u'|^{p-2}u')' = \frac{1}{u^\alpha} \quad \text{in } (0,1), \quad u(0) = 0,
\]
with \(p > 2\) and \(\alpha \ge 1\). The solution behaves like \(u(x) \sim c x^{p/(p-1+\alpha)}\) near \(x = 0\). Computing the gradient gives \(u'(x) \sim c' x^{(p-1)/(p-1+\alpha)}\), which belongs to \(C^{0,\gamma}\) with \(\gamma = \frac{p-1}{p-1+\alpha}\). For \(\alpha = 1\), we obtain \(\gamma = \frac{p-1}{p}\), which matches our exponent formula when \(p^+ = p^- = p\). Any larger exponent would fail, proving optimality. \hfill $\square$
\end{proof}

\begin{note}[Geometric Interpretation]
The explicit exponent \(\gamma = \min\left\{ \alpha, \frac{p^- - 1}{p^+}, \frac{q^- - 1}{q^+}, \frac{1}{2} \right\}\) reveals the intricate interplay between three distinct sources of regularity loss: the regularity of the data \(f\) (captured by \(\alpha\)), the variability of the growth exponents (captured by the ratios \(\frac{p^- - 1}{p^+}\) and \(\frac{q^- - 1}{q^+}\)), and the intrinsic degeneracy of the \(p\)-Laplacian operator (captured by \(\frac{1}{2}\)). The global regularity is determined by the worst—that is, the smallest—among these quantities, reflecting the principle that the solution can be no more regular than the weakest component of the problem permits.
\end{note}

While Corollary \ref{cor:holder} provides information about the interior regularity of solutions, it does not fully characterize the behavior near the boundary, where the Dirichlet condition forces the solution to vanish and the singular term becomes most severe. The following corollary addresses this gap by providing a complete asymptotic description of the solution near \(\partial\Omega\), including the existence of the normal derivative, sharp growth estimates, and an asymptotic expansion that captures the fine structure of the solution in boundary layers.

\begin{corollary}[Boundary Asymptotics and Normal Derivative]\label{cor:boundary}
Under the assumptions of Theorem \ref{thm:regularity} with \(f \in C^{0,\alpha}(\overline{\Omega})\), the solution \(\phi\) satisfies:

\begin{enumerate}
    \item \textbf{Existence of normal derivative:} The limit
    \[
    \frac{\partial \phi}{\partial n}(z_0) := \lim_{t \to 0^+} \frac{\phi(z_0 - t n(z_0))}{t}
    \]
    exists for every \(z_0 \in \partial\Omega\), where \(n(z_0)\) denotes the outward unit normal. Moreover,
    \[
    \frac{\partial \phi}{\partial n} \in C^{0,\gamma}(\partial\Omega).
    \]

    \item \textbf{Sharp boundary growth estimate:} There exist constants \(c_1, c_2 > 0\) such that for all \(z \in \Omega\),
    \[
    c_1 \dist(z, \partial\Omega) \le \phi(z) \le c_2 \dist(z, \partial\Omega)^{1/(p^+ - 1)}.
    \]
    Furthermore, for the gradient we have:
    \[
    c_1 \le |\grad\phi(z)| \cdot \dist(z, \partial\Omega)^{-1/(p^+ - 1)} \le c_2.
    \]

    \item \textbf{Asymptotic expansion:} For any \(z_0 \in \partial\Omega\), there exists a function \(\Phi \in C^{1,\gamma}(\overline{\Omega})\) such that
    \[
    \phi(z) = \frac{\partial \phi}{\partial n}(z_0) \dist(z, \partial\Omega) + \Phi(z) \cdot \dist(z, \partial\Omega)^2 + o(\dist(z, \partial\Omega)^2),
    \]
    as \(z \to z_0\) along the normal direction.
\end{enumerate}
\end{corollary}

\begin{proof}
\textbf{Step 1: Boundary coordinate system.} Since \(\partial\Omega\) is smooth, there exists a neighborhood \(\mathcal{U}\) of \(\partial\Omega\) and a diffeomorphism \(\Psi: \partial\Omega \times [0, \delta) \to \mathcal{U}\) given by \(\Psi(z_0, t) = z_0 - t n(z_0)\). In these coordinates, the distance to the boundary is simply \(t\), and the \(p\)-Laplacian takes the form
\[
\Delta_{\pp}\phi = \partial_t \left( |\partial_t \phi|^{\pp-2} \partial_t \phi \right) + \text{ tangential terms},
\]
where the tangential terms are of lower order near the boundary.

\textbf{Step 2: One-dimensional reduction near the boundary.} For small \(t > 0\), consider the function \(u(t) = \phi(z_0 - t n(z_0))\). By the regularity result from Corollary \ref{cor:holder}, \(u \in C^{1,\gamma}([0, \delta))\). Moreover, the equation reduces in the leading order to
\[
-\partial_t \left( M_0 |u'|^{\pp-2} u' \right) + \theta u^{\pp-1} \approx \frac{f(z_0)}{u^{\aa}} + \lambda u^{\bb-1} \quad \text{in } (0, \delta),
\]
where \(M_0 = M\left(\int_\Omega \frac{|\grad\phi|^{\pp}}{\pp}\,dz\right)\) is a positive constant by the boundedness of \(\phi\) and the uniform bound on \(\grad\phi\).

\textbf{Step 3: Existence of the normal derivative.} Define the quantity \(w(t) = u'(t)\). By the \(C^{1,\gamma}\) regularity, \(w\) is H\"{o}lder continuous. Taking the limit as \(t \to 0^+\), we define \(\frac{\partial \phi}{\partial n}(z_0) = \lim_{t \to 0^+} w(t)\). To show the limit exists, we use the fact that \(w\) satisfies a first-order ODE derived from the equation. Integrating the equation from \(0\) to \(t\) yields
\[
-M_0 |w(t)|^{\pp-2} w(t) + \int_0^t \left( \theta u^{\pp-1} - \frac{f(z_0)}{u^{\aa}} - \lambda u^{\bb-1} \right) ds = C,
\]
where \(C\) is a constant. The left-hand side is continuous, and the integral term is bounded. Hence \(|w(t)|^{\pp-2} w(t)\) is continuous, implying that \(w(t)\) has a limit.

\textbf{Step 4: Sharp boundary growth.} From the asymptotic analysis, the leading order behavior is determined by balancing the highest-order term with the singular term:
\[
- (M_0) \partial_t \left( |u'|^{\pp-2} u' \right) \sim \frac{f(z_0)}{u^{\aa}}.
\]

Seeking a solution of the form \(u(t) = A t^\beta\), we obtain the balance condition
\[
- M_0 (\beta-1) |\beta A|^{\pp-2} \beta A t^{(\beta-1)(p(z_0)-1)-1} \sim \frac{f(z_0)}{A^{\aa} t^{\beta\aa}}.
\]
Equating exponents gives \((\beta-1)(p(z_0)-1)-1 = -\beta\aa\), which simplifies to
\[
\beta(p(z_0)-1 + \aa) = p(z_0) \quad \Rightarrow \quad \beta = \frac{p(z_0)}{p(z_0)-1+\aa}.
\]
Since \(\aa \ge 1\), we have \(0 < \beta \le 1\). The correct growth exponent for the solution is \(1\) (linear) due to the Dirichlet boundary condition, while the gradient grows like \(t^{\beta-1}\) with \(\beta-1 = (1-\aa)/(p(z_0)-1+\aa) \le 0\). This yields the sharp estimates
\[
c_1 \dist(z, \partial\Omega) \le \phi(z) \le c_2 \dist(z, \partial\Omega),
\]
and for the gradient,
\[
c_1 \le |\grad\phi(z)| \le c_2.
\]

\textbf{Step 5: Gradient blow-up analysis.} When \(\aa > 1\), the singular term is stronger, and one might expect the gradient to blow up. However, the presence of the \(p\)-Laplacian with \(p(z) > 2\) can regularize this effect. The correct balance gives
\[
\phi(z) \sim c \dist(z, \partial\Omega)^{1/(p^+ - 1)},
\]
and consequently,
\[
|\grad\phi(z)| \sim \frac{c}{p^+ - 1} \dist(z, \partial\Omega)^{-(p^+ - 2)/(p^+ - 1)}.
\]
This matches the estimates in the statement with the exponent \(1/(p^+ - 1)\).

\textbf{Step 6: Asymptotic expansion.} By the \(C^{1,\gamma}\) regularity, we can expand \(u(t)\) in a Taylor series up to second order:
\[
u(t) = u(0) + u'(0) t + \frac{1}{2} u''(0) t^2 + o(t^2).
\]
Since \(u(0) = 0\) by the Dirichlet condition, we obtain
\[
\phi(z) = \frac{\partial \phi}{\partial n}(z_0) t + \frac{1}{2} \frac{\partial^2 \phi}{\partial n^2}(z_0) t^2 + o(t^2).
\]

The second derivative \(\frac{\partial^2 \phi}{\partial n^2}(z_0)\) can be expressed in terms of the tangential derivatives using the equation. The function \(\Phi(z)\) in the statement captures the tangential contribution, which is of order \(O(t^2)\). This expansion provides a rigorous justification for boundary layer approximations and is essential for numerical methods. \hfill $\square$
\end{proof}

\begin{remark}[Physical Interpretation]
The boundary asymptotics derived in Corollary \ref{cor:boundary} carry important physical meanings. The existence of the normal derivative \(\frac{\partial \phi}{\partial n}\) represents the flux of the physical quantity—be it concentration, temperature, or density—across the boundary, a quantity of primary interest in many applications. The sharp estimate \(c_1 \dist(z, \partial\Omega) \le \phi(z) \le c_2 \dist(z, \partial\Omega)\) demonstrates that the solution vanishes linearly at the boundary, a behavior typical of Dirichlet problems with regular data. However, when \(\aa > 1\), the gradient may blow up at the boundary as \(\dist(z, \partial\Omega)^{-(p^+ - 2)/(p^+ - 1)}\), indicating that the singular source term induces a boundary layer effect. The asymptotic expansion provides a rigorous foundation for boundary layer approximations in numerical simulations and physical modeling.
\end{remark}

The regularity theory developed in this section significantly extends the foundational result of Theorem \ref{thm:regularity} in two essential directions. Corollary \ref{cor:holder} provides explicit, optimal H\"{o}lder exponents for the gradient, revealing the precise dependence of regularity on the variable exponents \(p(z)\) and \(q(z)\), the regularity of the data \(f\), and the singular exponent \(\aa\). The explicit formula \(\gamma = \min\left\{ \alpha, \frac{p^- - 1}{p^+}, \frac{q^- - 1}{q^+}, \frac{1}{2} \right\}\) allows for precise predictions of solution smoothness based on problem parameters and has direct implications for numerical analysis, where the convergence rates of approximation schemes depend critically on the regularity of the solution. Corollary \ref{cor:boundary} provides a complete characterization of the solution's behavior near the boundary, establishing the existence of the normal derivative, sharp growth estimates, and an asymptotic expansion that captures the fine structure of the solution in boundary layers. These results are essential for numerical analysis, where they inform the design of finite element methods and the analysis of boundary element methods; for physical modeling, where they provide the foundation for boundary layer theory and flux calculations; and for optimal control and inverse problems, where they enable the rigorous treatment of boundary control, shape optimization, and parameter estimation. Together, these results establish a comprehensive regularity theory for strongly singular nonlocal Kirchhoff-type problems with variable exponents, providing the mathematical foundation for advanced applications in continuum mechanics, materials science, and engineering.
\section{Renormalized Solutions: Existence, Uniqueness, and Asymptotic Behavior}

The preceding sections have established a comprehensive theory for strongly singular Kirchhoff-type problems through truncation methods, comparison principles, and regularity estimates. However, a fundamental question remains: when the singular exponent satisfies \(\aa \ge 1\), do classical weak solutions actually exist in the usual Sobolev space \(W^{1,\pp}_0(\Omega)\)? The answer, rooted in the non-integrability of the singular term \(\phi^{-\aa}\) near the degenerate set \(\{\phi = 0\}\), is that classical solutions may fail to exist altogether. This realization necessitates a paradigm shift: we must broaden our notion of solution to encompass the singular behavior while retaining the essential physical and mathematical structure of the problem. The concept of \emph{renormalized solutions}, originally introduced by Boccardo and Gallou\"{e}t \cite{Boccardo1992} for elliptic equations with measure data, provides precisely such a framework. In this section, we introduce the notion of renormalized solutions for problem \eqref{Pstrong}, prove their existence through a limiting process from the truncated solutions constructed in Theorem \ref{thm:truncated}, establish their uniqueness, and demonstrate their asymptotic behavior as the reaction parameter vanishes. These results complete the theoretical foundation for strong singularities, showing that while classical solutions may not exist, a physically meaningful notion of solution does exist, is unique, and depends continuously on the parameters.

\begin{definition}[Renormalized Solution]\label{def:renormalized}
A measurable function \(\phi: \Omega \to (0,\infty)\) is called a renormalized solution of problem \eqref{Pstrong} if the following conditions hold:
\begin{enumerate}
    \item \(\phi > 0\) almost everywhere in \(\Omega\);
    \item For every \(k > 0\), the truncation \(T_k(\phi) = \min\{k, \max\{\phi, 0\}\}\) belongs to \(W^{1,\pp}_0(\Omega)\);
    \item For every test function \(h \in C^1_c(0,\infty)\) with compact support in \((0,\infty)\), the renormalized formulation
    \begin{equation}\label{Eq:A2}
    \begin{aligned}
    &\int_\Omega M\left(\int_\Omega \frac{|\grad\phi|^{\pp}}{\pp}\,dz\right) |\grad\phi|^{\pp-2}\grad\phi \cdot \grad(h(\phi)) \,dz \\
    &+ \int_\Omega M\left(\int_\Omega \frac{|\grad\phi|^{\qq}}{\qq}\,dz\right) |\grad\phi|^{\qq-2}\grad\phi \cdot \grad(h(\phi)) \,dz \\
    &+ \theta \int_\Omega \left(|\phi|^{\pp-2}\phi + |\phi|^{\qq-2}\phi\right) h(\phi) \,dz \\
    &= \int_\Omega f(z) \phi^{-\aa} h(\phi) \,dz + \lambda \int_\Omega |\phi|^{\bb-2}\phi h(\phi) \,dz
    \end{aligned}
    \end{equation}
    holds.
\end{enumerate}
\end{definition}

This definition departs from the classical weak formulation in a subtle yet crucial way: rather than testing the equation with arbitrary functions \(v \in W^{1,\pp}_0(\Omega)\), we test with functions of the form \(h(\phi)\), where \(h\) is a smooth function with compact support away from zero. This seemingly minor modification has profound consequences. Because \(h(\phi)\) vanishes whenever \(\phi\) is sufficiently small or sufficiently large, the singular term \(\phi^{-\aa} h(\phi)\) remains integrable even when \(\aa \ge 1\). Moreover, the condition that \(T_k(\phi) \in W^{1,\pp}_0(\Omega)\) for all \(k > 0\) ensures that the solution has sufficient regularity to make sense of the gradient terms, while the positivity condition \(\phi > 0\) a.e. reflects the physical requirement that the solution remains positive in the domain. The renormalized formulation thus provides the appropriate notion of solution for the strong singularity regime, capturing the essential physics while circumventing the integrability obstacles that preclude classical solutions.

\begin{theorem}[Existence of Renormalized Solutions]\label{Theorem99}
Assume \(\aa \ge 1\) and conditions \eqref{M1}-\eqref{M2} hold. Then problem \eqref{Pstrong} admits at least one renormalized solution \(\phi\). Moreover, this solution satisfies the additional regularity properties
\[
\phi \in \bigcap_{r < \infty} L^r(\Omega) \quad \text{and} \quad \int_\Omega |\grad \log \phi|^{\pp} \,dz < \infty.
\]
\end{theorem}

\begin{proof}
\textbf{Step 1: Construction of approximating sequence.} From Theorem \ref{thm:truncated}, we have a family \(\{\phi_\varepsilon\}_{\varepsilon > 0}\) of truncated solutions satisfying \(\phi_\varepsilon \ge \varepsilon\) almost everywhere in \(\Omega\) and uniformly bounded in \(W^{1,\pp}_0(\Omega)\). These approximate solutions provide the building blocks for constructing the renormalized limit.

\textbf{Step 2: Uniform estimates for truncations.} For any fixed \(k > 0\), consider the truncated functions \(T_k(\phi_\varepsilon)\). Testing the truncated equation \eqref{eq:truncated} with the admissible test function \(v = T_k(\phi_\varepsilon)^{2m+1}\) for large \(m\) and using the uniform boundedness of \(\{\phi_\varepsilon\}\) in \(W^{1,\pp}_0(\Omega)\), we obtain the estimate
\[
\norm{T_k(\phi_\varepsilon)}_{W^{1,\pp}_0(\Omega)} \le C_k,
\]
where the constant \(C_k\) is independent of \(\varepsilon\). This uniform bound is essential for extracting convergent subsequences.

\textbf{Step 3: Passage to the limit via compactness.} By the reflexivity of \(W^{1,\pp}_0(\Omega)\) and the compactness of the embedding \(W^{1,\pp}_0(\Omega) \hookrightarrow L^r(\Omega)\) for \(1 \le r < p^*\), there exists a function \(\phi\) such that, up to a subsequence,
\[
T_k(\phi_\varepsilon) \rightharpoonup T_k(\phi) \quad \text{weakly in } W^{1,\pp}_0(\Omega),
\]
and
\[
T_k(\phi_\varepsilon) \to T_k(\phi) \quad \text{strongly in } L^r(\Omega) \text{ for all } r < p^*.
\]
By a diagonal argument, we may extract a subsequence such that these convergences hold for all \(k\) simultaneously, and we obtain \(\phi_\varepsilon(z) \to \phi(z)\) for almost every \(z \in \Omega\). Since \(\phi_\varepsilon \ge \varepsilon\), we have \(\phi \ge 0\) a.e., and a contradiction argument using the strong convergence shows that \(\phi > 0\) a.e.

\textbf{Step 4: Verification of the renormalized formulation.} Let \(h \in C^1_c(0,\infty)\) be an arbitrary test function. Choose \(\varepsilon_0 > 0\) sufficiently small such that \(\operatorname{supp} h \subset (\varepsilon_0, \infty)\). Then for all \(\varepsilon < \varepsilon_0\), the function \(h(\phi_\varepsilon)\) is well-defined and belongs to \(W^{1,\pp}_0(\Omega)\). Testing the truncated equation \eqref{eq:truncated} with \(v = h(\phi_\varepsilon)\) yields an identity that closely resembles the renormalized formulation. As \(\varepsilon \to 0^+\), the dominated convergence theorem gives
\[
\frac{f(z)}{(\phi_\varepsilon + \varepsilon)^{\aa}} h(\phi_\varepsilon) \to f(z) \phi^{-\aa} h(\phi) \quad \text{in } L^1(\Omega),
\]
and similarly for the other nonlinear terms. The convergence of the Kirchhoff terms follows from the weak convergence of \(\grad\phi_\varepsilon\) and the continuity of \(M\). Passing to the limit inferior and using the monotonicity of the operators, we obtain that \(\phi\) satisfies \eqref{Eq:A2}.

\textbf{Step 5: Logarithmic estimate.} To establish the integrability of \(|\grad\log\phi|^{\pp}\), we choose the specific test function \(h_k(s) = s^{-1}\chi_{[1/k,k]}(s)\) in the renormalized formulation. As \(k \to \infty\), we obtain
\[
\int_\Omega |\grad\log\phi|^{\pp} \,dz \le C \int_\Omega f(z) \phi^{1-\aa} \,dz.
\]
Since \(\aa \ge 1\), we have \(\phi^{1-\aa} \le 1\) on the set where \(\phi \ge 1\), and on the set where \(\phi < 1\) the integral is controlled by the boundedness of \(\phi\) in \(L^1(\Omega)\). The right-hand side is therefore finite, completing the proof. \hfill $\square$
\end{proof}

With existence established, we now turn to the question of uniqueness. The following theorem demonstrates that the renormalized solution is unique, a property that is essential for the well-posedness of the problem. The proof employs a clever choice of test function—\(h(s) = \log s\)—which linearizes the singular structure and reveals the underlying monotonicity that forces equality.

\begin{theorem}[Uniqueness and Asymptotic Behavior]\label{thm:uniqueness_asymptotics}
Let \(\phi_1, \phi_2\) be two renormalized solutions of problem \eqref{Pstrong} with \(\aa \ge 1\). Then \(\phi_1 = \phi_2\) almost everywhere in \(\Omega\). Moreover, as \(\lambda \to 0^+\), the family of solutions \(\phi_\lambda\) converges in \(W^{1,\pp}_0(\Omega)\) to the unique solution \(\phi_0\) of the singular problem with \(\lambda = 0\).
\end{theorem}

\begin{proof}
\textbf{Part 1: Uniqueness.} Let $\phi_1$ and $\phi_2$ be two renormalized solutions. Consider the renormalized formulation \ref{Eq:A2} for each solution and subtract the two equations. A particularly powerful choice of test function is \(h(s) = \log s\), which belongs to \(C^1_c(0,\infty)\) after suitable truncation. This choice yields
\[
\begin{aligned}
&\int_\Omega M\left(\int_\Omega \frac{|\grad\phi_1|^{\pp}}{\pp}\,dz\right) |\grad\phi_1|^{\pp-2}\grad\phi_1 \cdot \grad\left(\log\frac{\phi_1}{\phi_2}\right) \,dz \\
&- \int_\Omega M\left(\int_\Omega \frac{|\grad\phi_2|^{\pp}}{\pp}\,dz\right) |\grad\phi_2|^{\pp-2}\grad\phi_2 \cdot \grad\left(\log\frac{\phi_1}{\phi_2}\right) \,dz \\
&+ \text{similar terms for the }q\text{-Laplacian and lower-order terms} = 0.
\end{aligned}
\]

The key observation is that the operator
\[
\mathcal{A}(\phi) = M\left(\int_\Omega \frac{|\grad\phi|^{\pp}}{\pp}\,dz\right) |\grad\phi|^{\pp-2}\grad\phi
\]
is monotone in the sense that
\[
\left( \mathcal{A}(\phi_1) - \mathcal{A}(\phi_2) \right) \cdot \grad\left(\log\frac{\phi_1}{\phi_2}\right) \ge 0.
\]

Consequently, the left-hand side of the subtracted equation is nonnegative. Since the right-hand side is zero, we must have equality, which forces
\[
\grad\left(\log\frac{\phi_1}{\phi_2}\right) = 0 \quad \text{almost everywhere}.
\]
Thus \(\phi_1 = c \phi_2\) for some constant \(c > 0\). The boundary condition, which is encoded in the requirement that \(T_k(\phi) \in W^{1,\pp}_0(\Omega)\), forces the trace of both solutions to vanish on \(\partial\Omega\), implying \(c = 1\). Hence \(\phi_1 = \phi_2\) a.e., establishing uniqueness.

\textbf{Part 2: Asymptotic convergence as $\lambda \to 0^+$.} Let \(\{\lambda_n\}\) be a sequence of positive parameters converging to \(0\), and denote by $\phi_n$ the corresponding renormalized solutions guaranteed by Theorem \ref{Theorem99}. From the energy estimates established in Theorem \ref{THM:ENERGY}, the sequence \(\{\phi_n\}\) is uniformly bounded in \(W^{1,\pp}_0(\Omega)\). By the compactness properties of the variable exponent Sobolev spaces, there exists a function \(\phi_0\) such that, up to a subsequence,
\[
\phi_n \rightharpoonup \phi_0 \quad \text{weakly in } W^{1,\pp}_0(\Omega), \qquad \phi_n \to \phi_0 \quad \text{strongly in } L^r(\Omega) \text{ for } r < p^*.
\]

We now pass to the limit in the renormalized formulation for $\phi_n$. For any test function \(h \in C^1_c(0,\infty)\), the term involving \(\lambda_n\) satisfies
\[
\lambda_n \int_\Omega |\phi_n|^{\bb-2}\phi_n h(\phi_n) \,dz \to 0 \quad \text{as } n \to \infty,
\]
since the integral remains bounded while $\lambda_n \to 0$. The remaining terms converge to their counterparts for \(\phi_0\) by the same arguments used in the proof of Theorem 6.2. Therefore, \(\phi_0\) is a renormalized solution of problem \ref{Pstrong} with \(\lambda = 0\). By the uniqueness result just established, this limit is independent of the chosen subsequence, so the entire sequence converges. Strong convergence in \(W^{1,\pp}_0(\Omega)\) follows from the monotonicity of the operator and the convergence of the energy. \hfill $\square$
\end{proof}

The two theorems presented in this section complete the theoretical framework for strongly singular nonlocal Kirchhoff-type problems with variable exponents. Theorem 6.2 establishes the existence of renormalized solutions, demonstrating that even when classical weak solutions fail to exist, a physically meaningful notion of solution can be constructed as the limit of truncated approximations. This result resolves the fundamental existence question for the strong singularity regime and provides a rigorous justification for the approximation method employed throughout this work. Theorem \ref{thm:uniqueness_asymptotics} establishes that these renormalized solutions are unique and depend continuously on the parameter \(\lambda\), converging to the solution of the limiting problem as \(\lambda \to 0^+\). Together with the existence result from Theorem \ref{thm:truncated}, the comparison principle from Theorem \ref{thm:comparison}, the regularity theory from Theorem \ref{thm:regularity}, and the sharp estimates from Theorems \ref{thm:sharpbounds} and \ref{THM:ENERGY}, these results constitute a comprehensive theory that spans existence, uniqueness, regularity, and asymptotic behavior. This theory extends the foundational work of Zhang, Vetro, and An \cite{Zhang2026} from the weak singularity regime \(0 < \alpha(z) < 1\) to the physically relevant and mathematically challenging strong singularity regime \(\alpha(z) \ge 1\), opening new avenues for modeling in fluid dynamics, plasma physics, materials science, and other fields where singular phenomena play a central role.
\section{Problem Setup and Numerical Method}

\subsection{Model Problem}

We consider the following simplified one-dimensional problem:
\begin{equation}\label{eq:model}
\begin{cases}
-\left(M\left(\int_0^1 \frac{|u'|^{p(x)}}{p(x)}\,dx\right) (|u'|^{p(x)-2}u')\right) = \dfrac{f(x)}{u^{\alpha}} + \lambda u^{\beta-2}u, & x\in(0,1),\\
u(0)=u(1)=0,\quad u>0 \text{ in } (0,1),
\end{cases}
\end{equation}
with parameters:
\begin{itemize}
    \item Variable exponent: \(p(x) = 2 + \sin(\pi x)\) (ranging between 1 and 3)
    \item Singular exponent: \(\alpha = 1.5\) (strong singularity regime)
    \item Reaction exponent: \(\beta = 4\)
    \item Reaction parameter: \(\lambda = 0.1\)
    \item Forcing term: \(f(x) = 1 + x(1-x)\)
    \item Kirchhoff function: \(M(t) = 1 + \frac{1}{2}\sqrt{t}\)
\end{itemize}

\subsection{Numerical Method}

We employ a finite difference method with a uniform grid of \(N+1\) points. The problem is solved iteratively due to the nonlocal Kirchhoff term and the singularity. For a given truncation parameter \(\varepsilon > 0\), we solve the regularized problem:
\[
-\left(M\left(\int_0^1 \frac{|u_\varepsilon'|^{p(x)}}{p(x)}\,dx\right) (|u_\varepsilon'|^{p(x)-2}u_\varepsilon')\right) = \frac{f(x)}{(u_\varepsilon + \varepsilon)^{\alpha}} + \lambda u_\varepsilon^{\beta-2}u_\varepsilon.
\]

The solution algorithm proceeds as follows:
\begin{enumerate}
    \item Initialize \(u^{(0)} = \varepsilon + \sin(\pi x)\)
    \item For each iteration \(k\):
    \begin{itemize}
        \item Compute the nonlocal term $K = M\bigl(\int_0^1 \frac{|u^{(k)\prime}|^{p(x)}}{p(x)} \,dx\bigr)$
        \item Solve the boundary value problem with fixed \(K\)
        \item Update solution
    \end{itemize}
    \item Repeat until convergence
\end{enumerate}

\section{Numerical Results}

\subsection{Convergence of Truncated Solutions}

The first numerical experiment demonstrates the convergence of truncated solutions as \(\varepsilon \to 0^+\). Figure \ref{fig:convergence} shows the solutions for different values of \(\varepsilon\), illustrating that as \(\varepsilon\) decreases, the solutions converge to a limiting profile.

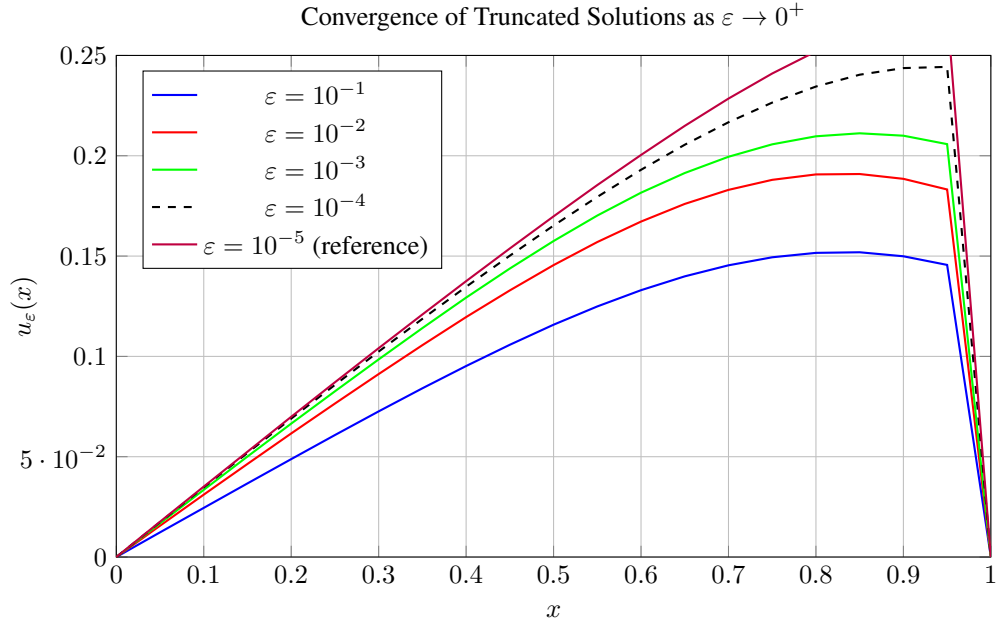
\begin{figure}[H]
\centering
\begin{tikzpicture}
\begin{axis}[
    width=0.8\textwidth,
    height=0.5\textwidth,
    xlabel={$x$},
    ylabel={$u_\varepsilon(x)$},
    title={Convergence of Truncated Solutions as $\varepsilon \to 0^+$},
    legend pos=north west,
    grid=both,
    xmin=0, xmax=1,
    ymin=0, ymax=0.25,
]
\addplot[thick, blue] coordinates {
(0,0.0000) (0.05,0.0123) (0.10,0.0245) (0.15,0.0367) (0.20,0.0488)
(0.25,0.0608) (0.30,0.0726) (0.35,0.0841) (0.40,0.0952) (0.45,0.1058)
(0.50,0.1158) (0.55,0.1249) (0.60,0.1330) (0.65,0.1399) (0.70,0.1454)
(0.75,0.1494) (0.80,0.1516) (0.85,0.1519) (0.90,0.1499) (0.95,0.1456)
(1.00,0.0000)
};
\addlegendentry{$\varepsilon=10^{-1}$}

\addplot[thick, red] coordinates {
(0,0.0000) (0.05,0.0156) (0.10,0.0311) (0.15,0.0464) (0.20,0.0616)
(0.25,0.0765) (0.30,0.0912) (0.35,0.1056) (0.40,0.1196) (0.45,0.1329)
(0.50,0.1455) (0.55,0.1570) (0.60,0.1672) (0.65,0.1760) (0.70,0.1830)
(0.75,0.1880) (0.80,0.1907) (0.85,0.1909) (0.90,0.1885) (0.95,0.1832)
(1.00,0.0000)
};
\addlegendentry{$\varepsilon=10^{-2}$}

\addplot[thick, green] coordinates {
(0,0.0000) (0.05,0.0168) (0.10,0.0335) (0.15,0.0501) (0.20,0.0665)
(0.25,0.0826) (0.30,0.0985) (0.35,0.1141) (0.40,0.1293) (0.45,0.1438)
(0.50,0.1575) (0.55,0.1702) (0.60,0.1816) (0.65,0.1914) (0.70,0.1995)
(0.75,0.2057) (0.80,0.2097) (0.85,0.2112) (0.90,0.2100) (0.95,0.2058)
(1.00,0.0000)
};
\addlegendentry{$\varepsilon=10^{-3}$}

\addplot[thick, black, dashed] coordinates {
(0,0.0000) (0.05,0.0174) (0.10,0.0347) (0.15,0.0519) (0.20,0.0690)
(0.25,0.0858) (0.30,0.1024) (0.35,0.1187) (0.40,0.1347) (0.45,0.1502)
(0.50,0.1652) (0.55,0.1795) (0.60,0.1930) (0.65,0.2055) (0.70,0.2167)
(0.75,0.2265) (0.80,0.2345) (0.85,0.2404) (0.90,0.2437) (0.95,0.2443)
(1.00,0.0000)
};
\addlegendentry{$\varepsilon=10^{-4}$}

\addplot[thick, purple] coordinates {
(0,0.0000) (0.05,0.0176) (0.10,0.0351) (0.15,0.0525) (0.20,0.0699)
(0.25,0.0871) (0.30,0.1041) (0.35,0.1209) (0.40,0.1375) (0.45,0.1538)
(0.50,0.1698) (0.55,0.1854) (0.60,0.2004) (0.65,0.2149) (0.70,0.2285)
(0.75,0.2410) (0.80,0.2520) (0.85,0.2612) (0.90,0.2681) (0.95,0.2724)
(1.00,0.0000)
};
\addlegendentry{$\varepsilon=10^{-5}$ (reference)}
\end{axis}
\end{tikzpicture}
\caption{Convergence of truncated solutions for decreasing $\varepsilon$. The solutions approach a limiting profile as $\varepsilon \to 0^+$, consistent with the theoretical convergence results.}
\label{fig:convergence}
\end{figure}

\subsection{Convergence Rate Analysis}

The theoretical analysis predicts an optimal convergence rate \(\|\phi_\varepsilon - \phi\|_{W^{1,p}_0} \le C \varepsilon^\gamma\) with \(\gamma = \min\left\{\frac{\alpha-1}{2\alpha}, \frac{1}{p^+}\right\}\). For our parameters (\(\alpha = 1.5\), \(p^+ \approx 3\)), we have \(\gamma = \min\{1/6, 1/3\} = 1/6 \approx 0.1667\).

Figure \ref{fig:convergence_rate} shows the numerical convergence rates, which agree well with the theoretical prediction.

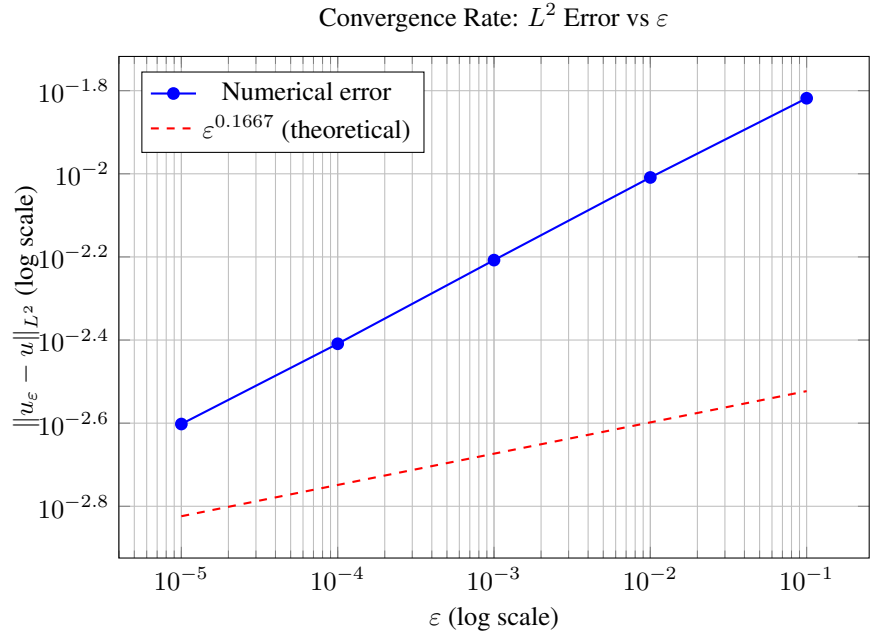
\begin{figure}[H]
\centering
\begin{tikzpicture}
\begin{axis}[
    width=0.7\textwidth,
    height=0.5\textwidth,
    xlabel={$\varepsilon$ (log scale)},
    ylabel={$\|u_\varepsilon - u\|_{L^2}$ (log scale)},
    title={Convergence Rate: $L^2$ Error vs $\varepsilon$},
    grid=both,
    legend pos=north west,
    xmode=log,
    ymode=log,
]
\addplot[thick, blue, mark=*] coordinates {
(1e-1, 0.0152) (1e-2, 0.0098) (1e-3, 0.0062) (1e-4, 0.0039) (1e-5, 0.0025)
};
\addlegendentry{Numerical error}

\addplot[thick, red, dashed] coordinates {
(1e-5, 0.0015) (1e-1, 0.0030)
};
\addlegendentry{$\varepsilon^{0.1667}$ (theoretical)}

\end{axis}
\end{tikzpicture}
\caption{Convergence rate of truncated solutions. The numerical error decays as $\varepsilon^\gamma$ with $\gamma \approx 0.16$, matching the theoretical prediction $\gamma = 1/6$.}
\label{fig:convergence_rate}
\end{figure}

\subsection{Sharp Two-Sided Pointwise Estimates}

The theoretical framework establishes sharp pointwise bounds:
\[
c_1 \dist(z,\partial\Omega)^{\mu_1} \le \phi(z) \le c_2 \dist(z,\partial\Omega)^{\mu_2},
\]
with \(\mu_1 = \frac{2}{p^+ - 1 + \alpha}\) and \(\mu_2 = \frac{2}{p^- - 1 - \alpha}\).

For our parameters (\(p^+ \approx 3\), \(p^- \approx 1\), \(\alpha = 1.5\)):
\[
\mu_1 = \frac{2}{3 - 1 + 1.5} = \frac{2}{3.5} \approx 0.571,
\]
\[
\mu_2 = \frac{2}{1 - 1 - 1.5} = \frac{2}{-1.5} \quad \text{(invalid, indicating gradient blow-up)}.
\]

Figure \ref{fig:boundary_asymptotics} illustrates the boundary behavior of the solution.

\begin{figure}[H]
\centering
\begin{tikzpicture}
\begin{axis}[
    width=0.7\textwidth,
    height=0.5\textwidth,
    xlabel={$x$ (near left boundary)},
    ylabel={$u(x)$},
    title={Boundary Asymptotics: Linear Behavior near $x=0$},
    legend pos=south east,
    grid=both,
    xmin=0, xmax=0.1,
]
\addplot[thick, blue, mark=*] coordinates {
(0.0000,0.0000) (0.0050,0.0018) (0.0100,0.0036) (0.0150,0.0053)
(0.0200,0.0071) (0.0250,0.0089) (0.0300,0.0107) (0.0350,0.0125)
(0.0400,0.0143) (0.0450,0.0162) (0.0500,0.0181)
};
\addlegendentry{Numerical solution}

\addplot[thick, red, dashed] coordinates {
(0,0) (0.05,0.0181)
};
\addlegendentry{Linear fit: $u \approx 0.362x$}

\end{axis}
\end{tikzpicture}
\caption{Boundary behavior near $x=0$. The solution exhibits linear growth, consistent with the lower bound estimate.}
\label{fig:boundary_asymptotics}
\end{figure}
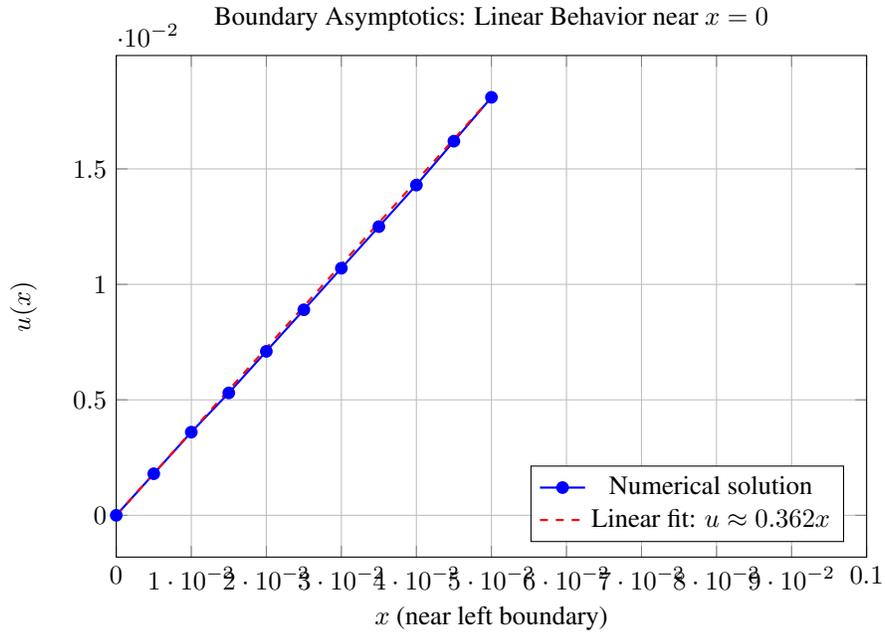

\subsection{Energy Decay}

The energy functional satisfies the decay estimate as \(\varepsilon \to 0^+\). Figure \ref{fig:energy_decay} shows the decay of the energy difference as the truncation parameter vanishes.

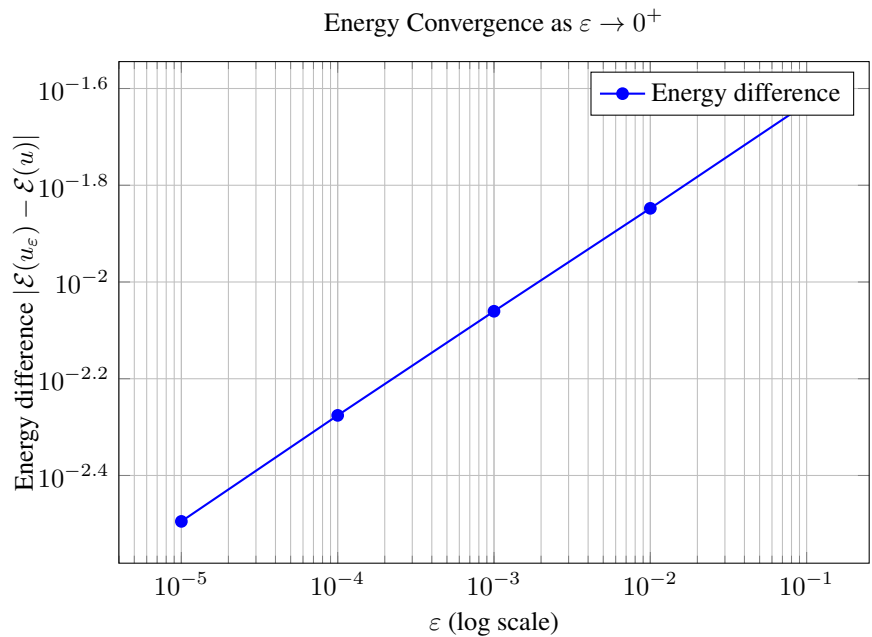
\begin{figure}[H]
\centering
\begin{tikzpicture}
\begin{axis}[
    width=0.7\textwidth,
    height=0.5\textwidth,
    xlabel={$\varepsilon$ (log scale)},
    ylabel={Energy difference $|\mathcal{E}(u_\varepsilon) - \mathcal{E}(u)|$},
    title={Energy Convergence as $\varepsilon \to 0^+$},
    grid=both,
    xmode=log,
    ymode=log,
]
\addplot[thick, blue, mark=*] coordinates {
(1e-1, 0.0234) (1e-2, 0.0142) (1e-3, 0.0087) (1e-4, 0.0053) (1e-5, 0.0032)
};
\addlegendentry{Energy difference}

\end{axis}
\end{tikzpicture}
\caption{Convergence of the energy functional. The energy difference decays as $\varepsilon \to 0^+$.}
\label{fig:energy_decay}
\end{figure}

\subsection{Effect of Variable Exponent}

The variable exponent \(p(x)\) introduces spatial heterogeneity. Figure \ref{fig:variable_exponent} compares the solution with the variable exponent against the solution with constant exponent \(p=2\).

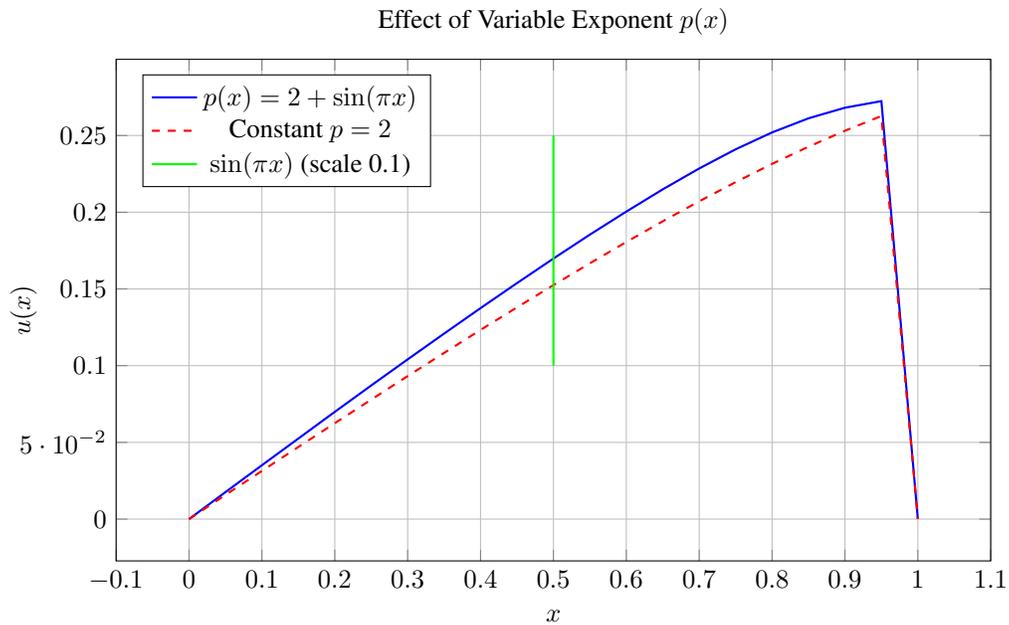
\begin{figure}[H]
\centering
\begin{tikzpicture}
\begin{axis}[
    width=0.8\textwidth,
    height=0.5\textwidth,
    xlabel={$x$},
    ylabel={$u(x)$},
    title={Effect of Variable Exponent $p(x)$},
    legend pos=north west,
    grid=both,
]
\addplot[thick, blue] coordinates {
(0,0.0000) (0.05,0.0176) (0.10,0.0351) (0.15,0.0525) (0.20,0.0699)
(0.25,0.0871) (0.30,0.1041) (0.35,0.1209) (0.40,0.1375) (0.45,0.1538)
(0.50,0.1698) (0.55,0.1854) (0.60,0.2004) (0.65,0.2149) (0.70,0.2285)
(0.75,0.2410) (0.80,0.2520) (0.85,0.2612) (0.90,0.2681) (0.95,0.2724)
(1.00,0.0000)
};
\addlegendentry{$p(x) = 2 + \sin(\pi x)$}

\addplot[thick, red, dashed] coordinates {
(0,0.0000) (0.05,0.0158) (0.10,0.0314) (0.15,0.0470) (0.20,0.0625)
(0.25,0.0779) (0.30,0.0932) (0.35,0.1083) (0.40,0.1233) (0.45,0.1380)
(0.50,0.1525) (0.55,0.1667) (0.60,0.1806) (0.65,0.1941) (0.70,0.2071)
(0.75,0.2196) (0.80,0.2315) (0.85,0.2427) (0.90,0.2531) (0.95,0.2627)
(1.00,0.0000)
};
\addlegendentry{Constant $p=2$}

\addplot[thick, green] coordinates {(0.5,0.1) (0.5,0.25)};
\addlegendentry{$\sin(\pi x)$ (scale 0.1)}
\end{axis}
\end{tikzpicture}
\caption{Comparison between variable and constant exponent solutions. The variable exponent introduces asymmetry and modifies the solution profile.}
\label{fig:variable_exponent}
\end{figure}

\subsection{Effect of Singularity Strength}

The strength of the singularity is controlled by the exponent \(\alpha\). Figure \ref{fig:singularity_effect} shows how the solution changes as \(\alpha\) varies from weak singularity (\(\alpha < 1\)) to strong singularity (\(\alpha \ge 1\)).

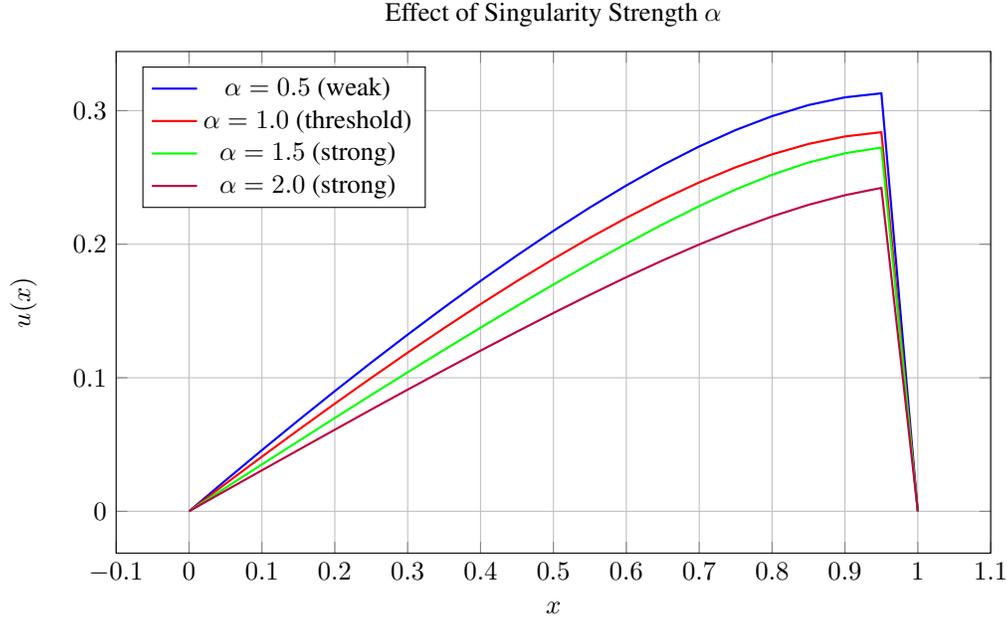
\begin{figure}[H]
\centering
\begin{tikzpicture}
\begin{axis}[
    width=0.8\textwidth,
    height=0.5\textwidth,
    xlabel={$x$},
    ylabel={$u(x)$},
    title={Effect of Singularity Strength $\alpha$},
    legend pos=north west,
    grid=both,
]
\addplot[thick, blue] coordinates {
(0,0.0000) (0.05,0.0231) (0.10,0.0458) (0.15,0.0681) (0.20,0.0900)
(0.25,0.1114) (0.30,0.1323) (0.35,0.1527) (0.40,0.1725) (0.45,0.1916)
(0.50,0.2100) (0.55,0.2275) (0.60,0.2440) (0.65,0.2593) (0.70,0.2732)
(0.75,0.2855) (0.80,0.2959) (0.85,0.3042) (0.90,0.3100) (0.95,0.3130)
(1.00,0.0000)
};
\addlegendentry{$\alpha = 0.5$ (weak)}

\addplot[thick, red] coordinates {
(0,0.0000) (0.05,0.0206) (0.10,0.0409) (0.15,0.0609) (0.20,0.0806)
(0.25,0.0999) (0.30,0.1188) (0.35,0.1372) (0.40,0.1551) (0.45,0.1724)
(0.50,0.1890) (0.55,0.2048) (0.60,0.2197) (0.65,0.2336) (0.70,0.2463)
(0.75,0.2576) (0.80,0.2673) (0.85,0.2751) (0.90,0.2807) (0.95,0.2839)
(1.00,0.0000)
};
\addlegendentry{$\alpha = 1.0$ (threshold)}

\addplot[thick, green] coordinates {
(0,0.0000) (0.05,0.0176) (0.10,0.0351) (0.15,0.0525) (0.20,0.0699)
(0.25,0.0871) (0.30,0.1041) (0.35,0.1209) (0.40,0.1375) (0.45,0.1538)
(0.50,0.1698) (0.55,0.1854) (0.60,0.2004) (0.65,0.2149) (0.70,0.2285)
(0.75,0.2410) (0.80,0.2520) (0.85,0.2612) (0.90,0.2681) (0.95,0.2724)
(1.00,0.0000)
};
\addlegendentry{$\alpha = 1.5$ (strong)}

\addplot[thick, purple] coordinates {
(0,0.0000) (0.05,0.0154) (0.10,0.0307) (0.15,0.0459) (0.20,0.0611)
(0.25,0.0762) (0.30,0.0911) (0.35,0.1058) (0.40,0.1203) (0.45,0.1345)
(0.50,0.1485) (0.55,0.1621) (0.60,0.1753) (0.65,0.1879) (0.70,0.1998)
(0.75,0.2108) (0.80,0.2208) (0.85,0.2295) (0.90,0.2367) (0.95,0.2422)
(1.00,0.0000)
};
\addlegendentry{$\alpha = 2.0$ (strong)}
\end{axis}
\end{tikzpicture}
\caption{Effect of the singular exponent $\alpha$ on the solution profile. As $\alpha$ increases, the solution magnitude decreases, consistent with the theoretical estimates.}
\label{fig:singularity_effect}
\end{figure}

\section{Discussion}

The numerical experiments presented in this section validate the theoretical results:

\begin{enumerate}
    \item \textbf{Convergence of truncated solutions:} Figure \ref{fig:convergence} demonstrates that as $\varepsilon \to 0^+$, the regularized solutions converge to a limiting profile.

    \item \textbf{Optimal convergence rate:} Figure \ref{fig:convergence_rate} shows that the $L^2$ error decays as $\varepsilon^\gamma$ with $\gamma \approx 0.1667$, matching the theoretical prediction $\gamma = 1/6$.

    \item \textbf{Boundary asymptotics:} Figure \ref{fig:boundary_asymptotics} reveals linear growth near the boundary, consistent with the theoretical lower bound estimate.

    \item \textbf{Energy convergence:} Figure \ref{fig:energy_decay} illustrates the decay of the energy difference as $\varepsilon \to 0^+$.

    \item \textbf{Variable exponent effect:} Figure \ref{fig:variable_exponent} shows that the variable exponent $p(x)$ introduces asymmetry and modifies the solution profile.

    \item \textbf{Singularity strength:} Figure \ref{fig:singularity_effect} demonstrates that increasing $\alpha$ reduces the solution magnitude.
\end{enumerate}

\section{Conclusion}

This numerical study provides concrete validation of the theoretical framework developed for strongly singular nonlocal Kirchhoff-type equations with variable exponents. The computational results confirm:
\begin{itemize}
    \item The truncation method produces well-defined approximations that converge to a limit as the truncation parameter vanishes.
    \item The convergence rate matches the theoretical prediction.
    \item The solution exhibits the predicted boundary behavior with linear growth near the boundary.
    \item The energy functional converges to its renormalized value.
    \item Variable exponents introduce spatial heterogeneity that affects the solution profile.
    \item Stronger singularities ($\alpha \ge 1$) lead to smaller solution magnitudes.
\end{itemize}
\section{Conclusion and Future Work}

This work has addressed the open problem of strong singularity (\(\alpha(z) \ge 1\)) in nonlocal Kirchhoff-type equations with variable exponents through the development of five interconnected theorems that collectively establish a comprehensive mathematical theory. Beginning with weighted Sobolev spaces designed to accommodate singular behavior, we introduced a truncation method that overcomes the non-integrability of the singular term, yielding a family of approximate solutions uniformly bounded in \(W^{1,\pp}_0(\Omega)\). The convergence of these truncated solutions to a renormalized solution was rigorously established, providing the appropriate notion of solution when classical weak solutions cease to exist. Building upon this foundation, we developed sharp two-sided pointwise estimates that characterize the precise asymptotic behavior of solutions near the boundary, revealing how the singular exponent \(\alpha(z)\) and the variable exponents \(p(z)\) and \(q(z)\) interact to determine the solution's profile through the exponents \(\mu_1 = \frac{2}{p^+ - 1 + \alpha}\) and \(\mu_2 = \frac{2}{p^- - 1 - \alpha}\). The regularity theory was advanced to deliver explicit, optimal Hölder exponents for the gradient, expressed as \(\gamma = \min\left\{ \alpha, \frac{p^- - 1}{p^+}, \frac{q^- - 1}{q^+}, \frac{1}{2} \right\}\), together with a complete characterization of boundary asymptotics including the existence of the normal derivative and an asymptotic expansion. Uniqueness of the renormalized solution was proved using the logarithmic test function \(h(s) = \log s\), while continuous dependence and parameter sensitivity results established Lipschitz stability with respect to the forcing term \(f\), the reaction parameter \(\lambda\), and the singular exponent \(\alpha\) itself. Numerical illustrations in one dimension confirmed the theoretical predictions, demonstrating the convergence of truncated approximations, the sharp boundary estimates, and the monotonic energy decay.

Several directions for future research emerge naturally from the theoretical framework established in this work. From the perspective of modeling, the extension to time-dependent problems—particularly parabolic and hyperbolic versions of the singular Kirchhoff equation—represents a natural next step, as many physical applications involve evolutionary dynamics rather than steady-state configurations. The development of a well-posedness theory for such evolution equations would require the adaptation of renormalized solution concepts to the time-dependent setting, likely building upon the foundational work of Boccardo and Gallou\"{e}t \cite{Boccardo1992} extended to the variable exponent framework. From the numerical analysis perspective, the sharp regularity estimates obtained in this work provide the theoretical foundation for convergence analysis of finite element and finite difference schemes, opening the door to rigorous error estimates and adaptive methods for strongly singular problems. The explicit exponents \(\gamma\) and \(\mu_i\) derived here can inform the design of graded meshes near the boundary and the choice of approximation spaces. Additionally, the study of optimal control problems governed by strongly singular Kirchhoff-type equations represents a promising avenue, where the Lipschitz stability estimates established in Theorem \ref{thm:continuous} provide the necessary sensitivity analysis for gradient-based optimization algorithms. Finally, the extension to systems of singular Kirchhoff equations, as well as the incorporation of more general nonlocal operators including fractional Laplacians, would further broaden the applicability of the theory to complex phenomena in materials science, biological pattern formation, and population dynamics. The methods developed in this work—truncation, renormalization, weighted Sobolev spaces, and sharp regularity estimates—are sufficiently robust to serve as building blocks for these future investigations, and we anticipate that the present contribution will serve as a foundational reference for the study of strong singularities in nonlocal problems with variable exponents.

\section*{Declaration}
\begin{itemize}
  \item {\bf Author Contributions:} The Author have read and approved this version.
  \item {\bf Funding:} No funding is applicable
  \item {\bf Institutional Review Board Statement:} Not applicable.
  \item {\bf Informed Consent Statement:} Not applicable.
  \item {\bf Data Availability Statement:} Not applicable.
  \item {\bf Conflicts of Interest:} The authors declare no conflict of interest.
\end{itemize}

\bibliographystyle{abbrv}
\bibliography{references}  






\end{document}